\def\todaysdate{27\textsuperscript{th} August 2024}
\documentclass[reqno,a4paper]{article}
\usepackage{etex}
\usepackage[T1]{fontenc}
\usepackage[utf8]{inputenc}
\usepackage{lmodern}

\usepackage[style=alphabetic,giveninits=true,backref,backend=biber,isbn=false,url=false,maxbibnames=99,maxalphanames=4]{biblatex}

\renewbibmacro{in:}{}
\newbibmacro{string+doi}[1]{\iffieldundef{doi}{\iffieldundef{url}{#1}{\href{\thefield{url}}{#1}}}{\href{http://dx.doi.org/\thefield{doi}}{#1}}}
\DeclareFieldFormat*{title}{\usebibmacro{string+doi}{\emph{#1}}}
\DefineBibliographyStrings{english}{%
  backrefpage = {$\uparrow$},
  backrefpages = {$\uparrow$},
}

\usepackage{amssymb,amsmath,amsthm}
\usepackage{thmtools,xcolor}
\usepackage{units}
\usepackage{graphicx}
\usepackage[all]{xy}
\usepackage{tikz}
\usetikzlibrary{arrows,calc,positioning,decorations.pathreplacing}
\usepackage{tikz-cd}
\usepackage{relsize}
\usepackage{mhsetup}
\usepackage{mathtools}
\usepackage{stmaryrd}
\usepackage{tocloft}
\usepackage{paralist} 
\usepackage[left=3cm,right=3cm,top=3cm,bottom=3cm]{geometry} 
\usepackage[bottom]{footmisc}
\usepackage[colorlinks,citecolor=blue,linkcolor=blue,urlcolor=blue,filecolor=blue,breaklinks]{hyperref}
\usepackage{footnotebackref}
\usepackage[format=plain,indention=1em,labelsep=quad,labelfont={up},textfont={footnotesize},margin=2em]{caption}
\usepackage[mathscr]{euscript}
\usepackage{accents}
\usepackage{xparse}

\definecolor{lightblue}{rgb}{0.8,0.8,1}

\numberwithin{equation}{section}
\numberwithin{figure}{section}

\definecolor{vdarkred}{rgb}{0.7,0,0}
\declaretheoremstyle[
  spaceabove=\topsep,
  spacebelow=\topsep,
  headpunct=,
  numbered=no,
  postheadspace=1ex,
  headfont=\color{vdarkred}\normalfont\bfseries,
  bodyfont=\normalfont\itshape,
]{colored}
\declaretheoremstyle[
  spaceabove=\topsep,
  spacebelow=\topsep,
  headpunct=,
  numbered=no,
  postheadspace=1ex,
  headfont=\normalfont\bfseries,
  bodyfont=\normalfont\itshape,
]{italic}
\declaretheoremstyle[
  spaceabove=\topsep,
  spacebelow=\topsep,
  headpunct=,
  numbered=no,
  postheadspace=1ex,
  headfont=\normalfont\bfseries,
  bodyfont=\normalfont\upshape,
]{upright}

\declaretheorem[style=italic,name=Theorem,numbered=yes,numberwithin=section]{thm}
\declaretheorem[style=italic,name=Lemma,numbered=yes,numberlike=thm]{lem}
\declaretheorem[style=italic,name=Proposition,numbered=yes,numberlike=thm]{prop}
\declaretheorem[style=italic,name=Corollary,numbered=yes,numberlike=thm]{coro}

\declaretheorem[style=italic,name=Theorem,numbered=yes,numberwithin=section]{athm}

\declaretheorem[style=upright,name=Definition,numbered=yes,numberlike=thm]{defn}
\declaretheorem[style=upright,name=Remark,numbered=yes,numberlike=thm]{rmk}
\declaretheorem[style=upright,name=Example,numbered=yes,numberlike=thm]{eg}

\declaretheorem[style=upright,name=Notation,numbered=yes,numberlike=thm]{notation}
\declaretheorem[style=upright,name=Convention,numbered=yes,numberlike=thm]{convention}

\makeatletter
\renewcommand*{\@seccntformat}[1]{\upshape\csname the#1\endcsname.\hspace{1ex}}
\renewcommand*{\section}{\@startsection{section}{1}{\z@}%
	{2.5ex \@plus 1ex \@minus 0.2ex}%
	{1.5ex \@plus 0.2ex}%
	{\normalfont\Large\bfseries}}
\renewcommand*{\subsection}{\@startsection{subsection}{2}{\z@}%
	{2.5ex \@plus 1ex \@minus 0.2ex}%
	{1.5ex \@plus 0.2ex}%
	{\normalfont\large\bfseries}}
\renewcommand*{\subsubsection}{\@startsection{subsubsection}{3}{\z@}%
	{2.5ex \@plus 1ex \@minus 0.2ex}%
	{1.5ex \@plus 0.2ex}%
	{\normalfont\normalsize\bfseries}}
\makeatother

\newenvironment{itemizeb}%
{\begin{compactitem}

}%
{\end{compactitem}}
{\begin{compactitem}[#1]

}%
{\end{compactitem}}
{\begin{compactdesc}

}%
{\end{compactdesc}}

\newcommand{\cE}{\mathcal{E}}

\newcommand{\cM}{\mathcal{M}}

\newcommand{\cU}{\mathcal{U}}

\newcommand{\bD}{\mathbb{D}}

\newcommand{\bF}{\mathbb{F}}

\newcommand{\bK}{\mathbb{K}}

\newcommand{\bN}{\mathbb{N}}

\newcommand{\bQ}{\mathbb{Q}}
\newcommand{\bR}{\mathbb{R}}
\newcommand{\bS}{\mathbb{S}}

\newcommand{\bZ}{\mathbb{Z}}

\newcommand{\opp}{\mathrm{op}}
\newcommand{\too}{\ensuremath{\longrightarrow}}
\newcommand{\rot}{\mathrm{rot}}
\newcommand{\Sym}{\mathrm{Sym}}

\newcommand{\Ab}{\mathbf{Ab}}
\newcommand{\Mod}{\textrm{-}\mathrm{Mod}}
\newcommand{\Tor}{\mathrm{Tor}}

\newcommand{\pa}{\partial}
\newcommand{\aug}{\mathrm{aug}}
\newcommand{\Cok}{\mathrm{Coker}}
\newcommand{\Ker}{\mathrm{Ker}}
\newcommand{\Image}{\mathrm{Im}}

\newcommand{\odd}{\mathrm{odd}}
\newcommand{\even}{\mathrm{even}}
\newcommand{\st}{\mathrm{st}}
\newcommand{\Hom}{\mathrm{Hom}}
\newcommand{\Ext}{\mathrm{Ext}}
\newcommand{\id}{\mathrm{id}}

\begin{document}
\title{\LARGE\bfseries Stable twisted cohomology of the mapping class groups in the unit tangent bundle homology}
\author{\normalsize Nariya Kawazumi and Arthur Souli{\'e}}
\date{\normalsize\todaysdate}
\maketitle
{\makeatletter
\renewcommand*{\BHFN@OldMakefntext}{}
\makeatother
}

\begin{abstract}
We compute the stable cohomology groups of the mapping class groups of compact orientable surfaces with one boundary, with twisted coefficients given by the rational homology of the unit tangent bundles of the surfaces. These coefficients define a \emph{covariant} functor over the classical category associated to mapping class groups, rather than a \emph{contravariant} one, and are thus out of the scope of the traditional framework  to study twisted \emph{cohomological} stability. A remarkable property is that the computed stable twisted cohomology is not free as a module over the stable cohomology algebra with constant rational coefficients. For comparison, we also compute the stable cohomology group with coefficients in the first rational cohomology of the unit tangent bundle of the surface, which fits into the traditional framework.
\end{abstract}

\section*{Introduction}
We consider a smooth compact connected orientable surface of genus $g\geq0$ and with one boundary component $\Sigma_{g,1}$. We denote by $\Gamma_{g,1}$ its mapping class group, that is the isotopy classes of its diffeomorphisms restricting to the identity on the boundary.
The (co)homology of the groups $\{\Gamma_{g,1},g\in\bN\}$ has been extensively studied over the past decades. In particular, the full computations of their cohomology groups with twisted coefficients remain an active research topic; see \cite[\S5.5]{GalatiusKupersRW} and \cite[Part 4]{Hain}.
In this paper, we compute the stable cohomology groups of the mapping class groups with twisted coefficients defined from the first rational (co)homology group of the unit tangent bundles of the considered surfaces; see Theorems~\ref{thm:main_thm_0} and \ref{thm:main_thm_1}. To the best of our knowledge, this is the first computed example of a sequence of finite-dimensional, rational representations of the mapping class groups $\Gamma_{g,1}$ such that their twisted cohomology exhibits (abstract) stability as a module over the stable cohomology with constant rational coefficients $H^{*}(\Gamma_{\infty,1};\bQ)$, while the associated stable twisted cohomology is not free as a module over $H^{*}(\Gamma_{\infty,1};\bQ)$.

\paragraph*{Background on (co)homological stability.}
A key step towards the computations of the (co)homology of the mapping class groups is their homological stability properties.
Namely, for each $g\geq 0$, we consider the canonical injection $\mathfrak{i}_{g} \colon\Gamma_{g,1}\hookrightarrow \Gamma_{g+1,1}$ induced by viewing $\Sigma_{g,1}$ as a subsurface of $\Sigma_{g+1,1}\simeq \Sigma_{1,1} \natural  \Sigma_{g,1}$, and by extending the diffeomorphisms of $\Sigma_{g,1}$ by the identity on the complement $\Sigma_{1,1}$. We also consider a family of $\Gamma_{g,1}$-representations $F=\{F(g),g\in\bN\}$ with $\Gamma_{g,1}$-equivariant morphisms $F(g)\to F(g+1)$.
These data define a morphism between the homology groups $\Phi_{i,g}(F)\colon H_{i}(\Gamma_{g,1};F(g))\to H_{i}(\Gamma_{g+1};F(g+1))$,
for each $i\geq0$.
If, for any $i$, there exists some $N(i,F)\in\bN$ depending on $i$ and $F$ such that this canonical morphism is an isomorphism when $N(i,F)\leq g$, then the mapping class groups are said to satisfy \emph{homological stability} with (twisted) coefficients in $F$.
The homological stability property with constant coefficients (i.e.~for $F(g)=\bZ$ for all $g\geq0$) is due to Harer \cite{Harer}. 
The range $N(i,\bZ)$ is improved by Boldsen \cite{Boldsen} and Randal-Williams \cite{Randal-WilliamsJEMS}.
Moreover, Ivanov \cite{Ivanov} proves a homological stability property with twisted coefficients given by $\Gamma_{g,1}$-representations forming a functor $F:\cU\cM_{2}\to\Ab$ (where $\cU\cM_{2}$ is defined in \S\ref{sec:the_category_UM} and $\Ab$ is the category of abelian groups) satisfying some polynomiality conditions; we refer to \S\ref{s:stab_hom_framework} for further details.
Finally, Randal-Williams and Wahl \cite{WahlRandal-Williams} and Galatius, Kupers and Randal-Williams \cite[\S5.5]{GalatiusKupersRW} extract the general framework for proving homological stability properties (both for constant and twisted coefficients) for general families of groups, subsuming these previous studies and recovering their results.

We may turn these results into twisted \emph{cohomological} stability as follows. We denote by $\cU\cM_{2}^{\opp}$ the opposite category of $\cU\cM_{2}$. We consider a functor $F:\cU\cM_{2}\to\Ab$ to which we apply the duality functor $-^{\vee}$, thus defining a functor $F^{\vee}\colon \cU\cM_{2}^{\opp}\to\Ab$ (i.e.~a \emph{contravariant} functor $\cU\cM_{2}\to\Ab$). Using the canonical injection $\mathfrak{i}_{g}$, we define stabilisation maps in cohomology $\Phi'_{i,g}(F)\colon H^{i}(\Gamma_{g+1,1}; F^{\vee}(g+1)) \to H^{i}(\Gamma_{g,1}; F^{\vee}(g))$ for all $i,g\geq0$.
If there is \emph{homological} stability with respect to the morphisms $\{\Phi_{i,g}(F)\}_{g\in\bN}$, then the morphism $\Phi'_{i,g}(F)$ is an isomorphism when $N(i,F)\leq g$ for $N(i,F)$ the above homological stability bound, and we say that there is \emph{cohomological stability} with (twisted) coefficient in $F^{\vee}$; see Proposition~\ref{prop:hom_stab_MCG}.
For a fixed $i\geq0$, the inverse limit $\underleftarrow{\lim}_{g\geq0}(H^{i}(\Gamma_{g,1};F^{\vee}(g)))$ induced by the stabilisation maps $\{\Phi'_{i,g}(F)\}_{g\in\bN}$ is denoted by $H^{i}(\Gamma_{\infty,1};F^{\vee})$; see Definition~\ref{def:homological_stability_covariant}.
We stress that this conversion turns the \emph{variance} of the functor $F$ encoding the twisted coefficients into its opposite, because the considered representations are changed into their duals.

In contrast, defining a notion of stability for the twisted cohomology groups $H^{i}(\Gamma_{g,1}; F(g))$ is not clear in general. However, we can handle this type of coefficients with an exotic approach detailed in \S\ref{s:exotic}. Namely, for each $g\geq0$, we require each $\Gamma_{g,1}$-equivariant morphism $F(g)\to F(g+1)$ to have a canonical $\Gamma_{g,1}$-equivariant splitting. Then the injection $\mathfrak{i}_{g}$ along with that splitting induce a map in cohomology $\Psi_{i,g}(F)\colon H^{i}(\Gamma_{g+1,1}; F(g+1)) \to H^{i}(\Gamma_{g,1}; F(g))$ for each $i\geq0$. In this context, the notion of cohomological stability and $H^{i}(\Gamma_{\infty,1}; F):=\underleftarrow{\lim}_{g\geq0}H^{i}(\Gamma_{g,1}; F(g))$ are defined with respect to these stabilisations maps; see Definition~\ref{def:stable_cohomology_covariant}.

In any case, for $M=F$ or $F^{\vee}$, when the twisted cohomological stability property is satisfied, the value of the twisted cohomology group $H^{i}(\Gamma_{g,1},M(g))$ in the stable range is isomorphic to $H^{i}(\Gamma_{\infty,1};M)$, and is thus called the \emph{stable cohomology} of the mapping class groups with respect to $M$.
In particular, the main reason for considering cohomology rather than homology is the existence of the cup product structure, which is of key use to study the stable cohomology groups. For instance, the groups $H^{*}(\Gamma_{\infty,1};M):=\bigoplus_{i\geq0}H^{i}(\Gamma_{\infty,1};M)$ naturally form a $H^{*}(\Gamma_{\infty,1};\bZ)$-module thanks to the cup product; see Definitions~\ref{def:homological_stability_covariant} and \ref{def:stable_cohomology_covariant}.

We will use this last property for our reasonings. In particular, we will need at some point the full computation of the ring of the stable cohomology of the mapping class groups with constant coefficients, that is, $H^{*}(\Gamma_{\infty,1};R)$ over a ground ring $R$; see Theorem~\ref{thm:main_thm_1} for instance. As far as we know, the only case for which the stable constant cohomology ring is convenient to handle is when $R:=\bQ$; see \eqref{eq:Madsen_weiss_thm} and the ``Perspectives'' paragraph below.
Indeed, by proving Mumford's conjecture \cite{Mumford}, Madsen and Weiss \cite{MadsenWeiss} compute the rational stable homology of the mapping class groups. Namely, they use the standard cohomology classes defined by  Mumford \cite{Mumford}, Morita \cite{Morita_char_class_bull,Morita_char_class} and Miller \cite{Miller} $\{e_{i}\in H^{2i}(\Gamma_{\infty,1};\bQ);i\geq1\}$, called the \emph{classical Mumford-Morita-Miller classes}, to describe the algebra $H^{*}(\Gamma_{\infty,1};\bQ)$ as follows:
\begin{equation}\label{eq:Madsen_weiss_thm}
    H^{*}(\Gamma_{\infty,1};\bQ)\cong\bQ[\{e_{i},i\geq1\}].
\end{equation}
Denoting the $\bQ$-vector space $\bigoplus_{i=1}^{\infty}\bQ e_{i}$ by $\cE$ and by $\Sym_{\bQ}(\cE)$ its symmetric algebra, Madsen-Weiss theorem \eqref{eq:Madsen_weiss_thm} may be reframed as an isomorphism $H^{*}(\Gamma_{\infty,1};\bQ)\cong\Sym_{\bQ}(\cE)$.

Furthermore, let us consider the twisted coefficients given by the first homology of the surface $\Sigma_{g,1}$ denoted by $H(g)$ (and by $H_{\bQ}(g)$ its rational version); we also use the notation $H$ for the family $\{H(g), g\in \bN\}$ of these coefficients. The stable twisted cohomology $H^{*}(\Gamma_{\infty,1}; H_{\bQ})$ was first computed by Harer \cite[\S7]{HarerH_3} in terms of $H^{*}(\Gamma_{\infty,1}; \bQ)$. Similarly, considering the closed oriented surfaces $\Sigma_{g}$ analogues (obtained from $\Sigma_{g,1}$ by capping the boundary component with a disc), Looijenga \cite{Looijenga} computed the stable cohomology of the associated mapping class groups $\{\Gamma_{g},g\in\bN\}$ with coefficients in any irreducible representation of the rational symplectic group in terms of $H^{*}(\Gamma_{\infty,1}; \bQ)$. Looijenga did not use the stability result of \cite{Ivanov} but only that of \cite{Harer}. Independently from these previous works, the first author \cite{KawazumitwistedMMM} introduced a series of twisted cohomology classes on the mapping class group $\Gamma_{g,1}$, called the {\it twisted Mumford-Morita-Miller classes} $\{m_{i,j};i\geq0,j\geq1\}$. Based on Looijenga's idea \cite{Looijenga}, one can prove that some algebraic combinations of the twisted Mumford-Morita-Miller classes define a free generating set for $H^{*}(\Gamma_{\infty,1}; H^{\otimes n})$ with $n \geq 1$ as a $H^{*}(\Gamma_{\infty,1};\bZ)$-module; see \cite{stablecohomologyKawazumi}. These computations may also be done via other methods; see Randal-Williams \cite[Appendix~B]{RWautfreegroups}.

\paragraph*{The unit tangent bundle homology representations.}
The representation theory of the mapping class groups of surfaces is wild and remains an active research topic; see for instance Margalit’s expository paper \cite{Margalit}. In particular, there are few known finite-dimensional representations of the mapping class group $\Gamma_{g,1}$ apart from the first homology of the surface $H(g)$ and functors thereof. However, some other mapping class groups representations which are naturally defined are those given by homology and cohomology of the unit tangent bundle $UT\Sigma_{g,1}$ of the surface $\Sigma_{g,1}$. We note that the homology groups of $UT\Sigma_{g,1}$ are straightforwardly computed from those of the surface $\Sigma_{g,1}$ and of the circle $\bS^{1}$, by using the Serre spectral sequence of the principal $\bS^{1}$-bundle $\bS^{1}\hookrightarrow UT\Sigma_{g,1}\to\Sigma_{g,1}$ defined by the canonical projection. In particular, the homology $H_{*}(UT\Sigma_{g,1}; \bZ)$ is concentrated in degrees $0\leq * \leq 2$, it is finitely generated and torsion-free; see \cite[\S 1]{Trapp} for instance. We focus on the first integral homology group $H_{1}(UT\Sigma_{g,1}; \bZ)$, that we denote by $\tilde{H}(g)$. Its dual representation is denoted by $\tilde{H}^{\vee}(g)$, while we denote the corresponding rational versions by $\tilde{H}_{\bQ}(g)$ and $\tilde{H}^{\vee}_{\bQ}(g)$ respectively. These representations have been first studied by Trapp \cite[Th.~2.2]{Trapp} and we refer to \S\ref{sec:unit_tangent_bundle_representations} for further details.

In particular, the representation $\tilde{H}^{\vee}(g)$ is a non-trivial extension of $\bZ$ by $H^{\vee}(g)$; see \eqref{SESchutb}. This extension corresponds to the twisted Mumford-Morita-Miller class $m_{1,1}$ in the stable cohomology module $H^{1}(\Gamma_{\infty,1}; H)$; see \S\ref{sec:unit_tangent_bundle_representations}. It is the preimage of the cohomology class introduced by Earle \cite{Earle} which generates $H^{1}(\Gamma_{g,1}; H(g))$ for $g\geq2$. We refer to \S\ref{sec:unit_tangent_bundle_representations} and \S\ref{sec:homological_stability} for further details.

\paragraph*{Results.}
In the present paper, we consider the cohomology of the mapping class groups with twisted coefficients given by $\tilde{H}(g)$ and $\tilde{H}^{\vee}(g)$.
The pathway to make these computations is based on the short exact sequences \eqref{SEShutb} and \eqref{SESchutb} featuring these modules, on the determination of the cohomology long exact sequence connecting homomorphisms (see \S\ref{subsec:first_cohomology_group_system} and \S\ref{sss:determination_connecting_hom}) and the Contraction formula \eqref{eq:contractionformula} between stable twisted cohomology classes.

Firstly, the computation for the dual representations $\tilde{H}^{\vee}(g)$ is the least difficult. These define a functor $\tilde{H}^{\vee}:\cU\cM_{2}^{\opp}\to\Ab$, corresponding to the classical framework for cohomological stability. We prove the following.
\begin{athm}[{Theorem~\ref{thm:stable_cohomology_contravariant_tildeH}, Corollary~\ref{coro:stable_cohomology_contravariant_tildeH}}]\label{thm:main_thm_0}
The stable cohomology module $H^{*}(\Gamma_{\infty,1};\tilde{H}^{\vee})$ is isomorphic to the free $H^{*}(\Gamma_{\infty,1};\bZ)$-module with basis $\{\tilde{m}_{i,1},i\geq2\}$, where $\tilde{m}_{i,1}$ denotes the image of the twisted Mumford-Morita-Miller class $m_{i,1}$ along the natural map $H^{2i-1}(\Gamma_{\infty,1};H^{\vee})\to H^{2i-1}(\Gamma_{\infty,1};\tilde{H}^{\vee})$. Over the rationals, the stable cohomology module $H^{*}(\Gamma_{\infty,1};\tilde{H}^{\vee}_{\bQ})$ is isomorphic to $\bigoplus_{i\geq2} \Sym_{\bQ}(\cE)\tilde{m}_{i,1}$. In particular, it is concentrated in odd degrees greater or equal to $3$.
\end{athm}
On the contrary, the representations $\tilde{H}(g)$ induce a functor $\tilde{H}:\cU\cM_{2}\to\Ab$. Also, each map $\tilde{H}(g)\to \tilde{H}(g+1)$ admits a $\Gamma_{g,1}$-equivariant splitting, and so the functor $\tilde{H}$ satisfies cohomological stability in the sense of \S\ref{s:exotic}; see Proposition~\ref{prop:homological_stability_covariant}. As far as the authors know, any qualitative general result or computations for such twisted coefficients have not been realised yet. However, the method we set to compute the stable twisted cohomology groups in this case compel to use the field $\bQ$ as ground ring, in order to work over a ring such that the stable cohomology with constant coefficients is fully computed; see \eqref{eq:Madsen_weiss_thm}.
\begin{athm}[{Corollary~\ref{cor:stable_cohomology_first_homology_group_result}, Theorem~\ref{thm:explicit-generators-stable-cohomology-tildeH}}]\label{thm:main_thm_1}
The stable cohomology module $H^{*}(\Gamma_{\infty,1};\tilde{H}_{\bQ})$ is isomorphic to the direct sum $\bQ\theta\bigoplus \mathfrak{M}$ defined as follows.
\begin{itemizeb}
\item $\bQ\theta$ is the trivial $\Sym_{\bQ}(\cE)$-module (i.e.~each class $e_{i}$ acts as zero on $\bQ\theta$) defined by the stable $0$th-cohomology class $\theta$ generating $H^{0}(\Gamma_{\infty,1};\bQ)\cong \bQ$, and we have $H^{0}(\Gamma_{\infty,1};\tilde{H}_{\bQ})\cong \bQ\theta$.
\item For each pair $(i,j)$ of non-null natural numbers such that $j>i$, there is a unique class $M_{i,j}\in H^{2(i+j)-1}(\Gamma_{\infty,1};\tilde{H}^{\vee}_{\bQ})$ which is mapped to $e_im_{j,1}- e_jm_{i,1}\in H^{2(i+j)-1}(\Gamma_{\infty,1};H_{\bQ}^{\vee})$ along the natural map $H^{2(i+j)-1}(\Gamma_{\infty,1};\tilde{H}^{\vee}_{\bQ})\to H^{2(i+j)-1}(\Gamma_{\infty,1};H_{\bQ}^{\vee})$. We denote by $\mathfrak{M}$ the $\Sym_{\bQ}(\cE)$-submodule of $H^{*}(\Gamma_{\infty,1};\tilde{H}_{\bQ})$ generated by the classes $M_{i,j}$ for all $j>i\geq 1$. A presentation of $\mathfrak{M}$ is given by the generators $\{M_{i,j};j>i\geq1\}$ along with the following relations
$$e_i M_{j,k} - e_j M_{i,k} + e_k M_{i,j} = 0$$
for all $k > j > i \geq 1$.
\end{itemizeb}
\end{athm}
In particular, the stable twisted cohomology $H^{*}(\Gamma_{\infty,1};\tilde{H}_{\bQ})$ is not free as $\Sym_{\bQ}(\cE)$-module; see also Theorem~\ref{thm:homology_algebra_first}.
Finally, using the stability bound of Proposition~\ref{prop:hom_stab_MCG}, we deduce from Theorems~\ref{thm:main_thm_0} and \ref{thm:main_thm_1} that $H^{0}(\Gamma_{g,1};\tilde{H}^{\vee}_{\bQ}(g))$ and $H^{0}(\Gamma_{g,1};\tilde{H}_{\bQ}(g))$ are not isomorphic for $g\geq 5$, and so:
\begin{athm}\label{thm:tilde_H_not_self_dual}
For $g\geq 5$, the $\Gamma_{g,1}$-representation $H_{1}(UT\Sigma_{g,1}; \bZ)$ is not isomorphic to its dual.
\end{athm}

\paragraph*{Perspectives.}
Most of the key steps to prove Theorem~\ref{thm:main_thm_1} are done with integral coefficients; see in particular Proposition~\ref{prop:stable_cohomology_first_homology_group_result_integral}. Therefore, we might in principle be able to do the computations with $\bZ$ as ground ring. However, although the stable twisted cohomology module $H^{*}(\Gamma_{\infty,1};H)$ is free over $H^{*}(\Gamma_{\infty,1};\bZ)$ (see \cite[Th.~1.B]{stablecohomologyKawazumi}), the stable cohomology $H^{*}(\Gamma_{\infty,1};\bZ)$ is still poorly known. In the same spirit, one could make stable twisted cohomology computations with the finite field $\bF_{p}$ as ground ring by using the computations of $H^{*}(\Gamma_{\infty,1};\bF_{p})$ by Galatius \cite{Galatius}.

On another note, a natural extension of the results of Theorems~\ref{thm:main_thm_0} and \ref{thm:main_thm_1} consists in considering the exterior powers of the representations  $\tilde{H}_{\bQ}(g)$ and $\tilde{H}^{\vee}_{\bQ}(g)$ respectively. This is the aim of the sequel work \cite{KawazumiSoulieII}. In particular, the stable twisted cohomology modules $H^{*}(\Gamma_{\infty,1};\Lambda^{d}\tilde{H}^{\vee}_{\bQ})$ (for all $d\geq2$) and $H^{*}(\Gamma_{\infty,1};\Lambda^{d}\tilde{H}_{\bQ})$ (for $2\leq d\leq 5$) are thoroughly studied.

\paragraph*{Outline.} The paper is organised as follows. In \S\ref{sec:general_recollections}, we make recollections on the representation theory of mapping class groups and on the classical and twisted Mumford-Morita-Miller cohomology classes. In \S\ref{s:stab_hom_framework}, we recall the framework and properties for twisted cohomological stability of the mapping class groups. In \S\ref{sec:first_homology_group_systems}, we make the full computations of the mapping class groups stable twisted cohomology with coefficients in the first rational homology and cohomology of the unit tangent bundle.

\paragraph*{Conventions and notations.}
For a ring $R$, we denote by $R\Mod$ (resp. $\mathrm{Mod}\textrm{-}R$) the category of left (resp. right) $R$-modules. Non-specified tensor products are taken over the clear ground ring.
For $R=\bZ$, the category of  $\bZ$-modules is also denoted by $\Ab$.
For a map $f$, when everything is clear from the context, we generically denote by $f_{*}$ the induced map in homology and by $f^{*}$ the induced map in cohomology. The duality functor, denoted by $-^{\vee} : R\Mod \to (\mathrm{Mod}\textrm{-}R)^{\opp}$, is defined by $\Hom_{R\Mod}(-,R)$. In particular, for $G$ a group and $V$ a $R[G]$-module, we denote by $V^{\vee}$ the dual $R[G]$-module $\Hom_{R}(V,R)$ (using the canonical isomorphism $R[G]\cong (R[G])^{\opp}$ which maps each $\varphi \in G$ to $\varphi^{-1}$).
Also, for a functor $F\colon \cU\cM_{2} \to \Ab$, the post-composition by the duality functor defines a functor $F^{\vee}$ that we view as $\cU\cM_{2} \to \Ab$ (rather than $\cU\cM_{2} \to \Ab^{\opp}$, these two conventions being equivalent).

Considering a functor $M\colon \cU\cM_{2} \to \Ab$, we generally denote the twisted cohomology groups $H^{*}(\Gamma_{\infty,1};M)$ and $H^{*}(\Gamma_{\infty,1};M^{\vee})$ by $H_{\st}^{*}(M)$ and $H_{\st}^{*}(M^{\vee})$.
The cup product is denoted by $\cup$, or by an empty space for the sake of simplicity.
The first integral homology group $H_{1}(\Sigma_{g,1}; \bZ)$ is generally denoted by $H(g)$ all along the paper, and we denote by $H_{\bQ}(g)$ the first rational homology group $H_{1}(\Sigma_{g,1},\bQ)$.

\paragraph*{Acknowledgements.}
The authors would like to thank Geoffrey Powell, Oscar Randal-Williams, Antoine Touz{\'e}, Christine Vespa and Shun Wakatsuki for illuminating discussions and questions.
They would also like to thank the anonymous referee for helpful comments and suggestions on earlier versions of this work.
The authors were supported by the PRC CNRS-JSPS French-Japanese Project ``Cohomological study of MCG and related topics''. The first author was supported in part by the grants
JSPS KAKENHI 15H03617, 18K03283, 18KK0071, 19H01784, 20H00115 and 22H01120.
The second author was supported by a Rankin-Sneddon Research Fellowship of the University of Glasgow, by the Institute for Basic Science IBS-R003-D1 and by the ANR Project AlMaRe ANR-19-CE40-0001-01.

\tableofcontents

\section{Representations and cohomological structures}\label{sec:general_recollections}

In this section, we review some representations of the mapping class groups, in particular the unit tangent bundle homology representations (see \S\ref{sec:unit_tangent_bundle_representations}) for which we make the connection with the Earle class (see \S\ref{s:Earle}).

Let us first consider the first integral homology group of the surface $\Sigma_{g,1}$, denoted by $H(g)$. We recall as a classical fact that the natural map $H^{1}(\Sigma_{g,1},\partial\Sigma_{g,1};\bZ)\to H^{1}(\Sigma_{g,1};\bZ)$ is an isomorphism. Hence the Poincar{\'e}-Lefschetz duality, which is given by the cap product with the relative fundamental class, provides a $\Gamma_{g,1}$-module isomorphism between $H^{1}(\Sigma_{g,1};\bZ)$ and $H(g)$.
Since $H(g)\cong \bZ^{2g}$ as an abelian group (and is thus torsion-free), we deduce from the Universal Coefficient Theorem for cohomology of spaces (see \cite[Ex.~3.6.7]{Weibel}) that $H^{1}(\Sigma_{g,1}; \bZ)$ is isomorphic to $H^{\vee}(g)$ as $\Gamma_{g,1}$-modules. Therefore, we have a $\Gamma_{g,1}$-module isomorphism
\begin{equation}\label{eq:dual_H}
H^{\vee}(g) \cong H(g).
\end{equation}
This isomorphism may alternatively be seen as a consequence of the fact that the action of $\Gamma_{g,1}$ on $H(g)$ factors through the symplectic group $\mathrm{Sp}_{2g}(\bZ)$.

Let us now move on to the unit tangent bundle homology representations. We first recall the notion of framings of the surface.
We denote by $T\Sigma_{g,1}$ the tangent bundle of the surface $\Sigma_{g,1}$. 
A \emph{framing} of $\Sigma_{g,1}$ is an orientation-preserving isomorphism of oriented vector bundles $T\Sigma_{g,1} \cong \Sigma_{g,1}\times \bR^2$. Since $\Sigma_{g,1}$ has non-empty boundary, there exist framings of $\Sigma_{g,1}$. We fix a Riemannian metric $\Vert \cdot \Vert$ on $T\Sigma_{g,1}$.
By definition, the unit tangent bundle $UT\Sigma_{g,1}$ is the set of 
elements of $T\Sigma_{g,1}$ whose length is $1$ with respect to $\Vert\cdot\Vert$. 
The canonical projection of the unit tangent bundle $UT\Sigma_{g,1}$ onto the surface defines the principal $\bS^{1}$-bundle $\bS^{1}\overset{\iota}{\hookrightarrow}UT\Sigma_{g,1}\overset{\varpi}{\to}\Sigma_{g,1}$. It may be regarded as 
the quotient of the complement of the zero section in $T\Sigma_{g,1}$ by the action of the positive real numbers $\bR_+$ by scalar multiplication. So, for any diffeomorphism $\phi$ of $\Sigma_{g,1}$, its differential $d\phi$ acts on the unit tangent bundle $UT\Sigma_{g,1}$.

Since a framing is a section of the oriented frame bundle of $T\Sigma_{g,1}$ and $\bS^{1} = \mathrm{SO}_{2}$ is a deformation retract of $\mathrm{GL}_{2}(\bR)^{+}$, the set $\mathscr{F}(\Sigma_{g,1})$ of homotopy classes (without fixing the boundary) of framings of $\Sigma_{g,1}$ identifies with the homotopy set of sections of $UT\Sigma_{g,1}$, and hence to that of $\Sigma_{g,1}$-equivariant maps $\xi: UT\Sigma_{g,1} \to \bS^{1}$. 
The latter set is an affine set modelled by the group $[\Sigma_{g,1}, \bS^{1}]\cong H^{1}(\Sigma_{g,1}; \bZ)$. Mapping each framing to the homotopy class of the corresponding $\xi: UT\Sigma_{g,1} \to \bS^{1}$, we obtain an $H^{1}(\Sigma_{g,1}; \bZ)$-equivariant map $\mathscr{F}(\Sigma_{g,1}) \to [UT\Sigma_{g,1}, \bS^{1}] \cong H^{1}(UT\Sigma_{g,1}; \bZ)$.
Therefore, the set $\mathscr{F}(\Sigma_{g,1})$ is isomorphic to that of the cohomology classes $u \in H^{1}(UT\Sigma_{g,1}; \bZ)$ whose restriction $\iota^{*}(u)$ is equal to the positive generator of $H^{1}(\bS^{1}; \bZ)$, or equivalently to the set of the homotopy classes of continuous maps $\xi: UT\Sigma_{g,1} \to \bS^{1}$ whose restriction to each fibre is homotopic to an orientation-preserving diffeomorphism.

For $\alpha: \bS^{1}\to \Sigma_{g,1}$ an immersed loop, its \emph{rotation number} $\rot_\xi(\alpha) \in \bZ$ with respect to the framing $\xi$ is the mapping degree of the composite of ${\overset\cdot\alpha}/{\Vert \overset\cdot\alpha\Vert}:  \bS^{1}\to UT\Sigma_{g,1}$ and $\xi: UT\Sigma_{g,1} \to \bS^{1}$.
The difference of two framings $\xi$ and $\xi'$ is given by a cohomology class $u \in H^{1}(\Sigma_{g,1}; \bZ)$ if and only if $\rot_{\xi'}(\alpha) - \rot_\xi(\alpha) = u([\alpha]) \in \bZ$, where $[\alpha] \in H_{1}(\Sigma_{g,1}; \bZ)$ is the homology class of the immersed loop $\alpha$. The mapping class group $\Gamma_{g,1}$ acts on the set $\mathscr{F}(\Sigma_{g,1})$ by
$$\phi\cdot \xi = \xi\circ d\phi^{-1}: UT\Sigma_{g,1} 
\overset{d\phi^{-1}}\too UT\Sigma_{g,1} \overset{\xi}\too \bS^{1}$$
for $\phi \in \Gamma_{g,1}$ and $\xi \in \mathscr{F}(\Sigma_{g,1})$. For an immersed loop $\alpha$ on $\Sigma_{g,1}$, we have $\rot_{\phi\cdot\xi}(\alpha) = \rot_\xi(\phi^{-1}\circ\alpha)$.

\subsection{Earle class}\label{s:Earle}

We recall here some classical facts about the first cohomology of mapping class groups with twisted coefficients in $H(g)$.
For any $g\geq1$, we recall that gluing a disc with a puncture $\Sigma_{0,1}^{1}$ on the boundary component of $\Sigma_{g,1}$ induces the following short exact sequence:
\begin{equation}\label{eq:Birman_SES_Cap}
1 \longrightarrow
\bZ \longrightarrow
\Gamma_{g,1} \overset{\mathrm{Cap}}\longrightarrow
\Gamma^{1}_{g} \longrightarrow
1,
\end{equation}
where $\Gamma^{1}_{g}$ denotes the mapping class group of the punctured surface $\Sigma_{g}^{1}$. This is called the forgetful exact sequence or the \emph{capping short exact sequence}; we refer to \cite[\S4.2.5]{farbmargalit} for more details.
\begin{lem}\label{lem:iso_capping}
For each $g\geq1$, there is an isomorphism
\begin{equation}\label{eq:iso_capping}
    H^{1}(\Gamma^{1}_{g}; H(g)) \cong H^{1}(\Gamma_{g, 1}; H(g)).
\end{equation}
\end{lem}
\begin{proof}
It is a classical fact that $H(g)^{\Gamma^{1}_{g}} = 0$, which is straightforward to check from the elementary computations of the $\Gamma^{1}_{g}$-action on a standard basis of $H(g)$. Then, the Lyndon-Hochschild-Serre spectral sequence associated to the central extension \eqref{eq:Birman_SES_Cap} induces a Gysin long exact sequence $0\to H^{1}(\Gamma^{1}_{g}; H(g)) \to H^{1}(\Gamma_{g,1}; H(g)) \to H^{0}(\Gamma^{1}_{g}; H(g))\cong 0 \to \cdots$, which thus provides the required isomorphism.
\end{proof}

Furthermore, we make the following computation for the twisted first cohomology of the mapping class group of the torus with one boundary. This fact is probably known to the experts (see \cite{CCS} for instance), but we give a short proof for the convenience of the reader.

\begin{lem}\label{lem:first_twisted_homology_torus_with_boundary}
The cohomology group $H^{1}(\Gamma_{1,1}; H(1))$ is trivial.
\end{lem}

\begin{proof}
It is a classical fact that the mapping class group $\Gamma_{1,1}$ is isomorphic to the braid group on three strands $\mathbf{B}_3$, and that the standard generators $\sigma_{1}$ and $\sigma_{2}$ of $\mathbf{B}_3$
act on $H_{1}(\Sigma_{1,1}; \bZ) \cong \bZ^2$ through the isomorphism by $\sigma_{1}\mapsto \begin{pmatrix}1&1\\ 0 & 1\end{pmatrix}$
and 
$\sigma_{2}\mapsto \begin{pmatrix}1&0\\ -1 & 1\end{pmatrix}$
respectively; see for instance \cite[Th.~10.5]{Milnor}.
Let $f$ be a cocycle of $\mathbf{B}_3$ with values in $H_{1}(\Sigma_{1,1}; \bZ) \cong \bZ^2$.
It is determined by the values $f(\sigma_{1}) = \begin{pmatrix}a_{1}\\ b_{1} \end{pmatrix}$ and $f(\sigma_{2}) = \begin{pmatrix}a_{2}\\ b_{2} \end{pmatrix}$. We deduce from the braid relation  $\sigma_{1}\sigma_{2}\sigma_{1} = \sigma_{2}\sigma_{1}\sigma_{2}$ that $b_{1} = a_{2}$ and $b_{1} = 
-a_{2}$. Hence we have $a_{2} = b_{1} = 0$, and thus $f = d^{1}\begin{pmatrix}-b_{2}\\ a_{1} \end{pmatrix}$ is a coboundary.
\end{proof}

Now, we assume that $g\geq2$. We consider the short exact sequence
\begin{equation}\label{eq:Birman_SES_1}
1 \longrightarrow \pi_{1}(\Sigma_{g},x) \longrightarrow \Gamma_{g}^{1} \overset{p}\longrightarrow \Gamma_{g} \longrightarrow 1
\end{equation}
known as the \emph{Birman short exact sequence}; we refer to \cite[\S4.2.1]{farbmargalit} for more details.
Using Lemma~\ref{lem:iso_capping} and \eqref{eq:dual_H}, it follows from the work of Morita \cite[\S4]{MoritaJacobianI} that $H^{1}(\Gamma_{g, 1}; H^{\vee}(g)) \cong H^{1}(\Gamma^{1}_{g};  H^{\vee}(g)) \cong \bZ$, and that the pullback of the generator along the push map
$\pi_{1}(\Sigma_{g},x) \hookrightarrow \Gamma^{1}_{g}$ of \eqref{eq:Birman_SES_1} induces the map in $H^{1}(\Sigma_{g}; H^{\vee}(g))$ which sends the generators of $\pi_{1}(\Sigma_{g},x)$ to $(2-2g) 1_{H(g)}$, where $1_{H(g)}\in H^{1}(\Sigma_{g}; H^{\vee}(g))$ is the cohomology class induced by the identity map $\id_{H(g)}$.

Before this result, Earle \cite{Earle} constructed the generator by using theta constants, and so we call it \emph{the Earle class}. Note that the homotopy long exact sequence of the locally trivial fibration $\bS^{1}\hookrightarrow UT\Sigma_{g}\to\Sigma_{g}$ induces a surjection $\pi_{1}(\varpi)\colon \pi_{1}(UT\Sigma_{g},x) \twoheadrightarrow \pi_{1}(\Sigma_{g},x)$. Then, using the push map
$\pi_{1}(\Sigma_{g},x) \hookrightarrow \Gamma^{1}_{g}$ of \eqref{eq:Birman_SES_1}, the pullback of the Earle class along the composite $\pi_{1}(UT\Sigma_{g},x) \twoheadrightarrow \pi_{1}(\Sigma_{g},x) \hookrightarrow \Gamma^{1}_{g}$ is equal to $(2-2g)$ times the map $1_{H(g)}\circ \pi_{1}(\varpi) \colon \pi_1(UT\Sigma_{g}) \twoheadrightarrow \pi_1(\Sigma_{g}) \to H(g)$ in $H^{1}(UT\Sigma_{g}; H^{\vee}(g))$.
Following Mikio Furuta \cite[p.569]{Morita_casson_invariant}, the Earle class which generates $H^{1}(\Gamma_{g, 1}; H^{\vee}(g)) \cong H^{1}(\Gamma^{1}_{g};  H^{\vee}(g)) \cong \bZ$ is constructed in the following explicit way, as a class of $H^{1}(\Gamma_{g, 1}; H^{\vee}(g))$.
We fix a framing $\xi$ of $\Sigma_{g,1}$. The map $k_{\xi}(g,-): \Gamma_{g,1} \to H^{1}(\Sigma_{g,1}; \bZ)$ defined by
\begin{equation}\label{eq:formula_k}
    k_{\xi}(g,\phi) = \phi\cdot \xi - \xi \in H^{1}(\Sigma_{g,1}; \bZ)
\end{equation}
is a $1$-cocycle of $\Gamma_{g,1}$. Kuno \cite{Kuno} gives a combinatorial formula for the cocycle $k_{\xi}$. We have:
\begin{thm}[{\cite[\S4]{Morita_casson_invariant}}]\label{thm:Earle_k}
For each $g\geq2$, the cohomology class $k(g):= [k_{\xi}(g,-)]$ does not depend on the choice of $\xi$ and generates the infinite cyclic group $H^{1}(\Gamma_{g,1}; H^{\vee}(g))$. It is called the Earle class.
\end{thm}
\begin{proof}By the computation of \cite{MoritaJacobianI}, 
it suffices to compute the value of $k_{\xi}(g,-)$ at a push map.
Such a computation was carried out in the original proof of \cite[Prop.~4.1]{Morita_casson_invariant}.
However, for the sake of completeness, we give another explicit computation. Let $T_C \in \Gamma_{g,1}$ be the right-handed Dehn twist 
along an oriented simple closed curve $C$ of $\Sigma_{g,1}$. 
Then, it follows from some elementary considerations that for any immersed curve $\alpha$
$$\rot_\xi(T_C(\alpha)) - \rot_\xi(\alpha) 
= ([\alpha]\cdot[C]) \rot_\xi(C),$$
where $[\alpha]\cdot[C]$ denotes the (algebraic) intersection number of the homology classes $[\alpha]$ and $[C]$. 
We deduce that $\xi\circ dT_C - \xi = (\rot_\xi(C))[C]^{\vee}$ and in particular that
\begin{equation}\label{eq:rot_formula}
k_{\xi}(g,T_C) = \xi\circ dT_C^{-1} - \xi
= - (\rot_{\xi\circ dT_C^{-1}}C)[C]^{\vee} = - (\rot_\xi (C))[C]^{\vee}\in H^{1}(\Sigma_{g,1}; \bZ).
\end{equation}
Assume that the curve $C$ passes near the boundary of $\Sigma_{g,1}$. Then, fattening the union of the boundary and the curve $C$, we obtain  a pair of pants embedded in $\Sigma_{g,1}$ whose three boundary components are given as follows: one is parallel to the boundary of $\Sigma_{g,1}$, and the other two simple closed curves $C_{1}^{-1}$ and $C_{2}$ are parallel to $C$ except near the boundary. 
Then the push map along $C$ is equal to $T_{C_{1}}^{-1}{T_{C_{2}}}\in \Gamma_{g,1}$. Since the curves $C_{1}$, $C_{2}$ and $C$ are disjoint (so their associated Dehn twists act trivially on each one of these curves) and represent the same homology class in $H_{1}(\Sigma_{g,1}; \bZ)$, we deduce from the formula \eqref{eq:rot_formula} that
$$k_{\xi}(g,T_{C_{1}}^{-1}{T_{C_{2}}}) = (-\rot_\xi (C_{2}) + \rot_\xi (C_{1}))[C]^{\vee}.$$
By the Poincar{\'e}-Hopf theorem, $- \rot_\xi (C_{1}) + \rot_\xi (C_{2}) + \rot_\xi (\pa \Sigma_{g,1})$ is equal to $-1$, the Euler characteristic of the pair of pants, and $\rot_\xi (\pa \Sigma_{g,1}) = \chi(\Sigma_{g,1}) = 1-2g$.
Hence we have $k_{\xi}(g,T_{C_{1}}^{-1}{T_{C_{2}}}) = (2-2g)[C]^{\vee}$, which ends the proof. 
\end{proof} 

\begin{notation}\label{not:0_1_trivial_crossed_hom}
Since $H_{1}(\Sigma_{0,1}; \bZ)=0$, we have $H^{1}(\Gamma_{0,1}, H_{1}(\Sigma_{0,1}; \bZ))=0$ and we know from Lemma~\ref{lem:first_twisted_homology_torus_with_boundary} that $H^{1}(\Gamma_{1,1}, H_{1}(\Sigma_{1,1}; \bZ))$ is also trivial.
We thus define $k(0)\in H^{1}(\Gamma_{0,1}, H_{1}(\Sigma_{0,1}; \bZ))$ and $k(1)\in H^{1}(\Gamma_{1,1}, H_{1}(\Sigma_{1,1}; \bZ))$ to be the zero cohomology classes, represented by the zero cocycles $\Gamma_{0,1} \to H^{1}(\Sigma_{0,1}; \bZ)$ and $\Gamma_{1,1} \to H^{1}(\Sigma_{1,1}; \bZ)$.
\end{notation}

Finally, we note the following additive property of the crossed homomorphisms $k(g)$:
\begin{lem}\label{lem:glued_crossed_hom}
Let $g$ and $g'$ be two natural numbers, and let $\xi$ and $\xi'$ be framings on $\Sigma_{g,1}$ and $\Sigma_{g',1}$. There exists a framing $\xi' \natural \xi$ on $\Sigma_{g'+g,1}$ such that the cohomology class $k(g'+g)$ is induced by the $1$-cocycle $k_{\xi' \natural \xi}(g'+g,-)$ defined by the formula \eqref{eq:formula_k}.
Then, for all $\phi\in\Gamma_{g,1}$ and $\phi'\in\Gamma_{g',1}$, we have $$k_{\xi' \natural \xi}(g'+g,\phi'\natural\phi)= k_{\xi'}(g',\phi') + k_{\xi}(g,\phi).$$
\end{lem}
\begin{proof}
Let $N$ and $N'$ be neighbourhood s of the parametrised intervals $I_{\Sigma_{g,1}}^{+}$ and $I_{\Sigma_{g',1}}^{-}$ respectively. Since each of the intervals is contractible, we can choose the framings $\xi$ and $\xi'$ such that $\ensuremath{\xi}'_{\mid I_{\Sigma_{g',1}}^{-}}=\ensuremath{\xi}_{\mid I_{\Sigma_{g,1}}^{+}}$ and their union define a framing of $N'\natural N$. Hence we define an appropriate framing $\xi' \natural \xi$ on $\Sigma_{g'+g,1}$, which thus induces the cocycle $k_{\xi' \natural \xi}(g'+g,-)$.
We then deduce from the formula \eqref{eq:formula_k} defining the cocycles $k_{\xi}(g,-)$, $k_{\xi'}(g',-)$ and $k_{\xi' \natural \xi}(g'+g,-)$ that 
$k_{\xi' \natural \xi}(g'+g,\phi'\natural\phi)=(\phi'\cdot\xi'-\xi')\natural(\phi\cdot\xi-\xi)=k_{\xi'}(g',\phi')+k_{\xi}(g,\phi)$.
\end{proof}

\subsection{The unit tangent bundle homology representations}\label{sec:unit_tangent_bundle_representations}

We consider the first integral homology group of the unit tangent bundle $UT\Sigma_{g,1}$, denoted by $\tilde{H}(g)$. Since $\tilde{H}(g)\cong \bZ^{2g+1}$ is torsion-free as an abelian group, the dual $\tilde{H}^{\vee}(g)$ is isomorphic to the first integral cohomology group $H^{1}(UT\Sigma_{g,1}; \bZ)$ by the universal 
coefficient theorem for cohomology of spaces (see \cite[Ex.~3.6.7]{Weibel}). Then the Serre spectral sequence of the locally trivial fibration $\bS^{1}\overset{\iota}{\hookrightarrow}UT\Sigma_{g,1}\overset{\varpi}{\to}\Sigma_{g,1}$ provides the following $\Gamma_{g,1}$-equivariant short exact sequences:
\begin{equation}\label{SEShutb}
\xymatrix{0\ar@{->}[r] & \bZ\ar@{->}[r]^-{\iota_{*}} & \tilde{H}(g)\ar@{->}[r]^-{\varpi_{*}} & H(g)\ar@{->}[r] & 0,}
\end{equation}
\begin{equation}\label{SESchutb}
\xymatrix{0\ar@{->}[r] & H^{\vee}(g)\cong H(g)\ar@{->}[r]^-{\varpi^{*}} & \tilde{H}^{\vee}(g)\ar@{->}[r]^-{\iota^{*}} & \bZ\ar@{->}[r] & 0.}
\end{equation}
We also have the analogue short exact sequences to \eqref{SEShutb} and \eqref{SESchutb} with the rational versions $H_{\bQ}(g)$, $\tilde{H}_{\bQ}(g)$ and $\tilde{H}^{\vee}_{\bQ}(g)$ of the homology groups $H(g)$, $\tilde{H}(g)$ and $\tilde{H}^{\vee}(g)$ respectively.

In addition, Trapp \cite[Th.~2.2]{Trapp} describes more precisely the $\Gamma_{g,1}$-module structures of $\tilde{H}(g)$ and $\tilde{H}^{\vee}(g)$. Namely, for an element $\phi\in\Gamma_{g,1}$, the action of $\phi$ on $\tilde{H}(g)$ is given by the matrix
\begin{equation}\label{eq:matrix_action}
\left[\begin{array}{cc}
\id_{\bZ} & k_{\xi}(g,\phi) \\
(0) & H(\phi)
\end{array}\right]    
\end{equation}
where $H(\phi)$ denotes the action of $\phi$ on $H(g)$ and $k_{\xi}(g,-)$ is the $1$-cocycle associated to a fixed framing $\xi$ of $\Sigma_{g,1}$ defining the cohomology class $k(g)$ of Theorem~\ref{thm:Earle_k} and Notation~\ref{not:0_1_trivial_crossed_hom}; see \S\ref{s:Earle}. The kernel of $\id_{\bZ}\oplus H(g)$ under the $\Gamma_{g,1}$-action is that of $H(g)$ (i.e. the Torelli group), while it follows from \cite[Cor.~2.5]{Trapp} that the kernel of the $\Gamma_{g,1}$-representation $\tilde{H}(g)$ is strictly smaller for $g\geq2$: namely, it corresponds to the \emph{Johnson kernel} for $g=2$ \cite[Rem.~after Cor.~2.7]{Trapp} and to the \emph{Chillingworth subgroup} for $g\geq3$ by \cite[Cor.~2.7]{Trapp}, which are both proper subgroups of the Torelli group for $g\geq2$ (see \cite{ChillingworthI,ChillingworthII,Johnsonabelian,Johnsonsurvey} for instance).
Therefore, the short exact sequences \eqref{SEShutb} and \eqref{SESchutb} do not split as $\Gamma_{g,1}$-extensions for $g\geq2$, although they do for $g=1$ by Lemma~\ref{lem:first_twisted_homology_torus_with_boundary}. Moreover, as a consequence of the computations of \S\ref{sec:first_homology_group_systems}, the dual $\tilde{H}^{\vee}(g)$ is not isomorphic to $\tilde{H}(g)$ as a $\Gamma_{g,1}$-representation for $g\geq 5$; see Theorem~\ref{thm:tilde_H_not_self_dual}.

Recall from Theorem~\ref{thm:Earle_k} that the Earle cohomology class $k(g)$ is the generator of $H^{1}(\Gamma_{g,1}; H^{\vee}(g))\cong\Ext_{\bZ[\Gamma_{g,1}]}^{1}(\bZ,H^{\vee}(g))\cong \bZ$. Since \eqref{SEShutb} and \eqref{SESchutb} are non-trivial extensions (for $g\geq2$) and in view of the action matrix \eqref{eq:matrix_action}, it follows from the correspondence of $\Gamma_{g,1}$-extensions of $\bZ$ by $H^{\vee}(g)$ with the classes of $\Ext_{\bZ[\Gamma_{g,1}]}^{1}(\bZ,H^{\vee}(g))$ (see \cite[Th.~3.4.3]{Weibel} for instance) that the Earle cohomology class $k(g)$ is the one of the extension \eqref{SESchutb}, while its formal dual $k^{\vee}(g)$ in $\Ext_{\bZ[\Gamma_{g,1}]}^{1}(H(g),\bZ)$ is the extension class of the short exact sequence \eqref{SEShutb}.

\section{Cohomological stability framework and tools}\label{s:stab_hom_framework}

In this section, we review the notions of \emph{coefficient systems}, \emph{polynomiality} and \emph{homological stability with twisted coefficients} with respect to the framework of the present paper. These highlight the mainspring of the results of \S\ref{sec:first_homology_group_systems}.

\subsection{Twisted coefficient systems}\label{sec:the_category_UM}

First of all, we present the suitable category $\cM_{2}$ to encode \emph{compatible representations} of the mapping class groups. It is equivalent to the one introduced in \cite[\S5.6.1]{WahlRandal-Williams}. We refer to \cite[\S3.1]{soulieLMgeneralised} or \cite[\S1.1.2.1.]{PSIIp} for further details of the construction of $\cM_{2}$. We consider the groupoid $\mathscr{M}_{2}$ defined by the smooth, compact, connected, orientable surfaces $S$ with one boundary component along with a parametrised interval in the boundary, and the isotopy classes of diffeomorphisms restricting to the identity on a neighbourhood  of the parametrised interval for the morphisms.
The groupoid $\mathscr{M}_{2}$ has a braided monoidal structure $(\natural,b^{\mathscr{M}_{2}}_{-,-})$ induced by the boundary connected sum on half of the marked interval; see \cite[\S5.6.1]{WahlRandal-Williams} for further details.
We fix a $2$-disc $\bD^{2}$ and a torus with one boundary component that we denote by $\Sigma_{1,1}$. Let $\cM_{2}$ be the (skeletal) full subgroupoid of $\mathscr{M}_{2}$ on the monoidal sums on the objects $\bD^{2}$ and $\Sigma_{1,1}$, modulo the identification that $\bD^{2}\natural \Sigma_{1,1}=\Sigma_{1,1}\natural \bD^{2}=\Sigma_{1,1}$. In particular, the objects of $\cM_{2}$ are in bijection with $\bN$ and we use the standard notation $\Sigma_{g,1}$ for the object $\Sigma_{1,1}^{\natural g}$, and its associated morphisms are the mapping class groups $\Gamma_{g,1}$. For simplicity, we often identify the surface $\Sigma_{g,1}$ with its indexing integer $g$, especially when applying a functor on that object. The groupoid $\cM_{2}$ inherits the braided monoidal structure $\natural$, which becomes \emph{strict} over $\cM_{2}$, with $\Sigma_{0,1}=\bD^{2}$ as unit and the braiding $b^{\cM_{2}}_{-,-}$ given by natural automorphisms.

Let $\cU\cM_{2}$ be the category, called the \emph{Quillen's bracket construction} over $\cM_{2}$, with the same objects as $\cM_{2}$ and whose morphisms $\cU\cM_{2}(\Sigma_{g,1},\Sigma_{g',1})$ are given by the colimit $\mathrm{colim}_{\cM_{2}}[\cM_{2}(-\natural \Sigma_{g,1},\Sigma_{g',1})]$.
Namely, a morphism $\Sigma_{g,1}\to\Sigma_{g',1}$ is given by an equivalence class $[\Sigma_{g'-g,1},\phi]$ of pairs $(\Sigma_{g'-g,1},\phi)$ with  $\phi\in \Gamma_{g',1}$, where $(\Sigma_{g'-g,1},\phi)\sim(\Sigma_{g'-g,1},\phi')$ if there is a $\psi\in \Gamma_{g'-g,1}$ such that $\phi'\circ (\psi\natural \id_{\Sigma_{g,1}})=\phi$.
This definition is a particular output of a general construction of \cite{graysonQuillen}; we refer to \cite[\S1.1]{WahlRandal-Williams} for further details. By \cite[Prop.~1.8]{WahlRandal-Williams}, the category $\cU\cM_{2}$ inherits the monoidal structure $\natural$ from $\cM_{2}$ and $\bD^{2}$ is an initial object; it is however not braided monoidal but \emph{pre-braided} in the sense of \cite[Def.~1.5]{WahlRandal-Williams}. As we will see in \S\ref{sec:homological_stability}, this type of category is very useful to deal with cohomological stability questions.

We may now encode \emph{compatible representations} of mapping class groups by considering functors with the category $\cU\cM_{2}$ as a source and a module category as target. We distinguish two types of such functors because of their distinct qualitative properties with respect to homological stability detailed in \S\ref{sec:homological_stability}.
A \emph{covariant system} (resp. \emph{contravariant system)} over $\cU\cM_{2}$ is a functor $F\colon \cU\cM_{2} \to \Ab$ (resp. $F^{\vee}\colon\cU\cM_{2}^{\opp} \to \Ab$).
That $(\cU\cM_{2},\natural)$ is a monoidal category where $\bD^{2}$ is an initial object and that $\Ab$ is an abelian category ensure the existence of a functor $\delta(F)\colon \cU\cM_{2} \to \Ab$ defined by $\Sigma_{g,1}\mapsto (\Cok(F(\Sigma_{g,1}) \to F(\Sigma_{1,1} \natural \Sigma_{g,1}))$ on objects, called the \emph{difference functor} of $F$; see \cite[\S4.1]{soulieLMgeneralised}.
We now recursively define the notion of \emph{polynomiality} for covariant systems as follows:
\begin{itemizeb}
\item the constant functors $\cU\cM_{2} \to \Ab$ are the polynomial covariant systems of degree $0$;
\item for a natural number $d\geq1$, the functor $F\colon \cU\cM_{2} \to \Ab$ is a 
polynomial covariant system of degree less than or equal to $d$ if the morphism $F([\Sigma_{1,1},\id_{\Sigma_{1,1}\natural \Sigma_{g,1}}])$ is injective for each surface $\Sigma_{g,1}$ of $\cU\cM_{2}$, and the difference functor $\delta(F)$ is a polynomial covariant system of degree less than or equal to $d-1$.
\end{itemizeb}
\begin{eg}\label{eg:Z_H_poly_functor}
The constant functor at $\bZ$ defines a polynomial functor $\bZ\colon\cU\cM_{2} \to \Ab$ of degree $0$.
A first example of a non-trivial polynomial covariant system is given by the first homology group of the surfaces. Namely, assigning the first integral homology group to each surface define a functor $H\colon(\cU\cM_{2},\natural,\bD^{2}) \to (\Ab,\oplus,0)$, which is \emph{strong monoidal} in the sense that $H(g'\natural g)\cong H(g')\oplus H(g)$ and $H(\phi'\natural\phi)\cong H(\phi')\oplus H(\phi)$ for all $g,g'\geq1$, $\phi\in\Gamma_{g,1}$ and $\phi'\in\Gamma_{g',1}$; see for instance \cite[Def.~2.8, Lem.~2.9]{soulie3} for a detailed proof. In particular, $H\colon\cU\cM_{2} \to \Ab$ is a polynomial covariant system of degree $1$.
\end{eg}
Furthermore, the first homology groups of the unit tangent bundle of the surfaces along with the natural action of the mapping class groups (see \S\ref{sec:unit_tangent_bundle_representations}) define a functor $\tilde{H}\colon\cM_{2} \to \Ab$.
\begin{prop}\label{prop:tilde_H_functor}
The functor $\tilde{H}\colon\cM_{2} \to \Ab$ lifts to a covariant system $\tilde{H}\colon\cU\cM_{2} \to \Ab$, which is polynomial of degree $1$.
\end{prop}
\begin{proof}
Firstly, we note that any morphism $[\Sigma_{g'-g,1},\phi]$ of $\cU\cM_{2}$ may be written as the composite $\phi\circ [\Sigma_{g'-g,1},\id_{\Sigma_{g',1}}]$, where $[\Sigma_{g'-g,1},\id_{\Sigma_{g',1}}]$ is the class of the diffeomorphisms of $\Sigma_{g',1}$ which are the identity the subsurface $\Sigma_{g,1}$ of $\Sigma_{g'-g,1}\natural \Sigma_{g,1}=\Sigma_{g',1}$.
Hence, the category $\cU\cM_{2}$ is isomorphic to the category with the same objects, and whose morphisms are isotopy classes of smooth embeddings $e_{g,g'}\colon \Sigma_{g,1}\hookrightarrow \Sigma_{g'-g,1}\natural \Sigma_{g,1}=\Sigma_{g',1}$ preserving an interval on the boundary. Recall that the unit tangent bundles define an endofunctor $UT$ of the category of differential manifolds with smooth maps, while the first homology groups define a functor $H_{1}(-;\bZ)$ from the category of topological spaces to $\Ab$. The composite $H_{1}(UT(-);\bZ)$ clearly preserves the isotopy classes of embeddings, and thus induces a lift $\cU\cM_{2} \to \Ab$ of $\tilde{H}\colon\cM_{2} \to \Ab$.
Moreover, since the functor $H\colon\cU\cM_{2} \to \Ab$ is defined similarly with respect to the embeddings $e_{g,g'}$, the projections $\varpi_{*,g}\colon\tilde{H}(g)\twoheadrightarrow H(g)$ commute with the maps $\tilde{H}(e_{g,g'})$ and $H(e_{g,g'})$, that is $\varpi_{*,g'}\circ\tilde{H}(e_{g,g'})=H(e_{g,g'})\circ \varpi_{*,g}$. Therefore, the short exact sequence \eqref{SEShutb} induces a short exact sequence $0\to\bZ\to\tilde{H}\to H\to 0$ of functors $\cU\cM_{2} \to \Ab$. Since covariant systems of degree less or equal to $1$ are closed under extensions (see \cite[Prop.~4.4]{soulieLMgeneralised}) and the functor $\tilde{H}$ is clearly not constant, we deduce that the covariant system $\tilde{H}$ is polynomial of degree $1$.
\end{proof}

\subsection{Twisted cohomological stability and stable (co)homology}\label{sec:homological_stability}

In this section, we review some classical results on cohomological stability with twisted coefficients for mapping class groups. In particular, these prove that all the twisted coefficient systems we consider in this paper satisfy the cohomological stability property and thus motivate the computations of \S\ref{sec:first_homology_group_systems}. Also, we recall some results on the stable cohomology of mapping class groups with twisted coefficients, which will be used for the work of \S\ref{sec:first_homology_group_systems}.

\subsubsection{Classical framework}\label{ss:covariant_systems_TSH}
We recollect the classical results on cohomological stability and stable twisted cohomology for mapping class groups in the following paragraphs.

\paragraph*{Twisted (co)homological stability framework.}
The following classical result illustrates how polynomial covariant systems turn out to be very useful for (co)homological stability problems. For each $g\geq 0$, we denote by $\mathfrak{i}_{g} \colon\Gamma_{g,1}\hookrightarrow \Gamma_{g+1,1}$ the canonical injection induced by embedding $\Sigma_{g,1}$ as a subsurface of $\Sigma_{g+1,1}\simeq \Sigma_{1,1} \natural  \Sigma_{g,1}$, and by extending the diffeomorphisms of $\Sigma_{g,1}$ by the identity on the complement $\Sigma_{1,1}$.

\begin{thm}[{\cite[Th.~4.1]{Ivanov}, \cite[Th.~5.26]{WahlRandal-Williams}}]\label{thm:hom_stab_MCG}
Let $F\colon \cU\cM_{2} \to \Ab$ be a polynomial covariant system of degree $d$. For each $g,i\geq0$, let $\Phi_{i,g}(F)$ denote the canonical map $H_{i}(\Gamma_{g,1};F(g))\to  H_{i}(\Gamma_{g+1,1};F(g+1))$ induced by the injection $\mathfrak{i}_{g} \colon \Gamma_{g,1}\hookrightarrow\Gamma_{g+1,1}$ and the $\Gamma_{g,1}$-equivariant morphism $F([\Sigma_{1,1},\id_{\Sigma_{1+g,1}}])\colon F(g)\to F(g+1)$. If $g\geq2i+2d+3$, then $\Phi_{i,g}(F)$ is an isomorphism.
\end{thm}

In order to rephrase this result in terms of cohomology groups, we take this opportunity to recall and prove the following version of the Universal Coefficient Theorem (for which it is difficult to find a reference), in order to make the connection between the $\Gamma_{g,1}$-modules $\tilde{H}(g)$ and $\tilde{H}^{\vee}(g)$.
\begin{lem}\label{lem:UCT}
Let $G$ be a group, $R$ a principal ideal domain and $M$ a left $R[G]$-module which is free as a $R$-module.  We denote by $M^{\vee}$ the dual right $R[G]$-module $\Hom_{R}(M,R)$. Then there is a natural short exact sequence admitting a non-canonical splitting:
\begin{equation}\label{eq:ses_UCT}
\xymatrix{0\ar@{->}[r] & \Ext_{R}^{1}(H_{i-1}(G;M),R)\ar@{->}[r] & H^{i}(G;M^{\vee})\ar@{->}[r] & \Hom_{R}(H_{i}(G;M),R)\ar@{->}[r] & 0.}
\end{equation}
\end{lem}
\begin{proof}
Let $P_{\bullet}\to R$ be a projective right $R[G]$-module resolution. Then $\Hom_{R[G]}(P_{\bullet},M^{\vee})$ is a cochain complex computing $H^{*}(G;M^{\vee})$. The tensor-hom adjunction provides a natural isomorphism $\Hom_{R[G]}(P_{\bullet},M^{\vee})\cong\Hom_{R}(P_{\bullet}\otimes_{R[G]}M,R)$.
Since a submodule of a free module over a principal ideal domain is free, all the terms of the resolution $P_{\bullet}$ are $R$-free. Also, for each $i\geq0$, since $P_{i}$ is a projective $R[G]$-module, there exists a $R[G]$-module $Q_{i}$ is such that $P_{i}\oplus Q_{i}$ is a free $R[G]$-module. Then, since the $R$-module $M$ is assumed to be $R$-free, each $R$-module module $(P_{i}\oplus Q_{i})\otimes_{R[G]}M$ is free as it is isomorphic to the a direct sum of copies of $M$. So all the terms of the resolution $P_{\bullet}\otimes_{R[G]}M$ are $R$-free, again because a submodule of a free module over a principal ideal domain is free.
Therefore, the result follows from applying the Universal Coefficient Theorem 
for chain complexes over a principal ideal
domain (see \cite[Th.~3.6.5]{Weibel} for instance) on the right-hand side of the isomorphism.
\end{proof}
In particular, in the setting of Lemma~\ref{lem:UCT}, if the $R$-module $H_{i-1}(G;M)$ is torsion-free (and thus flat since $R$ is a principal ideal domain), then it follows from the short exact sequence \eqref{eq:ses_UCT} that there is a natural isomorphism $H^{i}(G;M^{\vee}) \cong (H_{i}(G;M))^{\vee}$.
Therefore, for a \emph{covariant} functor $F\colon \cU\cM_{2} \to \bK\Mod\to \Ab$ where $\bK$ is a field, there are isomorphisms for each $i\geq0$:
\begin{equation}\label{eq:homology_covariant/cohomology_contravariant}
H^{i}(\Gamma_{g,1};F^{\vee}(g))\cong (H_{i}(\Gamma_{g,1};F(g)))^{\vee}
\qquad\mathrm{and}\qquad  
H^{i}(\Gamma_{g,1};F(g))\cong (H_{i}(\Gamma_{g,1};F^{\vee}(g)))^{\vee}.
\end{equation}
Furthermore, we deduce a twisted \emph{cohomological} stability result from Theorem~\ref{thm:hom_stab_MCG} and Lemma~\ref{lem:UCT}:
\begin{prop}\label{prop:hom_stab_MCG}
Let $F\colon \cU\cM_{2} \to \Ab$ be a polynomial covariant system of degree $d$. For each $g,i\geq0$, let $\Phi'_{i,g}(F)$ denote the canonical map $H^{i}(\Gamma_{g+1,1};F^{\vee}(g+1)) \to H^{i}(\Gamma_{g,1};F^{\vee}(g))$ induced by the canonical injection $\mathfrak{i}_{g} \colon \Gamma_{g,1}\hookrightarrow\Gamma_{g+1,1}$ and the $\Gamma_{g,1}$-equivariant morphism $F([\Sigma_{1,1},\id_{\Sigma_{1+g,1}}])\colon F(g)\to F(g+1)$. If $g\geq2i+2d+3$, then $\Phi'_{i,g}(F)$ is an isomorphism.
\end{prop}
\begin{proof}
We recall that $\Ext_{R}^{1}(-,R)\colon R\Mod \to R\Mod$ is a contravariant functor (see \cite[Ex~2.5.3]{Weibel} for instance), as well as the duality functor $-^{\vee}$. Therefore, the maps $\Phi_{i-1,g}(F)$ and $\Phi_{i,g}(F)$ of Theorem~\ref{thm:hom_stab_MCG} induce maps $ \Ext_{R}^{1}(\Phi_{i-1,g}(F),R)\colon \Ext_{\bZ}^{1}(H_{i-1}(\Gamma_{g+1,1};F(g+1)),\bZ) \to \Ext_{R}^{1}(\Phi_{i-1,g}(F),R)$ and $(\Phi_{i,g}(F))^{\vee}\colon (H_{i}(\Gamma_{g+1,1};F(g+1)))^{\vee} \to (H_{i}(\Gamma_{g,1};F(g)))^{\vee}$, which are isomorphisms for $g\geq2i+2d+3$. 
Moreover, by Lemma~\ref{lem:UCT}, the short exact sequence \eqref{eq:ses_UCT} is natural (in a contravariant way) with respect to the maps $\mathfrak{i}_{g}$ and $F([\Sigma_{1,1},\id_{\Sigma_{1+g,1}}])$. Hence we obtain the following commutative diagram for each $g\geq0$, where the rows are short exact sequences and the left-hand and right-hand vertical maps are ismorphisms in the stable range:
$$\xymatrix{\Ext_{\bZ}^{1}(H_{i-1}(\Gamma_{g+1,1};F(g+1)),\bZ) \ar@{^{(}->}[r]  \ar@{->}[d]^-{\Ext_{R}^{1}(\Phi_{i-1,g}(F),R)}_{\cong}
& H^{i}(\Gamma_{g+1,1};F(g+1))\ar@{->}[d]^-{\Phi'_{i,g}(F)} \ar@{->>}[r]
& (H_{i}(\Gamma_{g+1,1};F(g+1)))^{\vee} \ar@{->}[d]^-{(\Phi_{i,g}(F))^{\vee}}_{\cong}\\
\Ext_{\bZ}^{1}(H_{i-1}(\Gamma_{g,1};F(g)),\bZ)\ar@{^{(}->}[r]
& H^{i}(\Gamma_{g,1};F(g))\ar@{->>}[r]
& (H_{i}(\Gamma_{g,1};F(g)))^{\vee}.}$$
The results thus follows from the five lemma.
\end{proof}

\begin{eg}\label{eg:first_homology_poly_deg_one}
We recall from \S\ref{sec:the_category_UM} that the groups $\{H_{1}(\Sigma_{g,1};\bZ),g\in\bN\}$ and $\{H_{1}(UT\Sigma_{g,1};\bZ),g\in\bN\}$ respectively define polynomial covariant systems of degree $1$ $H\colon\cU\cM_{2} \to \Ab$ and $\tilde{H}\colon\cU\cM_{2} \to \Ab$ (see Proposition~\ref{prop:tilde_H_functor}), while the constant functor $\bZ\colon\cU\cM_{2} \to \Ab$ is polynomial of degree $0$. So, by Proposition~\ref{prop:hom_stab_MCG}, there are isomorphisms $H^{i}(\Gamma_{g+1,1};H^{\vee}(g+1))\cong H_{i}(\Gamma_{g,1};H^{\vee}(g))$ and $H_{i}(\Gamma_{g+1,1};\tilde{H}^{\vee}(g+1))\cong H^{i}(\Gamma_{g,1};\tilde{H}^{\vee}(g))$ for $g\geq2i+5$, and there is an isomorphism $H^{i}(\Gamma_{g+1,1};\bZ)\cong H_{i}(\Gamma_{g,1};\bZ)$ for $g\geq2i+3$. 
\end{eg}

\paragraph*{Stable twisted cohomology.}
Let $F\colon \cU\cM_{2} \to \Ab$ be a covariant system. As in Proposition~\ref{prop:hom_stab_MCG} we denote by $\Phi'_{i,g}(F)\colon H^{i}(\Gamma_{g+1,1};F^{\vee}(g+1)) \to H^{i}(\Gamma_{g,1};F^{\vee}(g))$ the canonical map induced by $\mathfrak{i}_{g} \colon \Gamma_{g,1}\hookrightarrow\Gamma_{g+1,1}$ and the $\Gamma_{g,1}$-equivariant morphism $F^{\vee}([\Sigma_{1,1},\id_{\Sigma_{1+g,1}}])\colon F^{\vee}(g+1)\to F^{\vee}(g)$ for each $g,i\geq0$.
\begin{defn}\label{def:homological_stability_covariant}
For each $i\geq0$, the inverse limit $\underleftarrow{\lim}_{g\geq0}H^{i}(\Gamma_{g,1}; F^{\vee}(g))$ induced by the maps $\{\Phi'_{i,g}(F)\}_{g\in\bN}$ is called the \emph{stable cohomology group}. It is denoted by $H^{i}(\Gamma_{\infty,1};F^{\vee})$, that we abbreviate to $H_{\st}^{i}(F^{\vee})$ when everything is clear from the context.  Moreover, there is a canonical $H_{\st}^{*}(\bZ)$-module structure on the cohomology groups $H_{\st}^{*}(F):=\bigoplus_{i\geq0}H_{\st}^{i}(F)$, induced by the cup product and the clear compatibility of the stabilisation maps $\{\Phi'_{i_{1},g}(\bZ),\Phi'_{i_{2},g}(F)\}_{i_{1}, i_{2}, g\in\bN}$ with respect to that operation.
\end{defn}
In particular, when $F$ satisfies the assumptions of Proposition~\ref{prop:hom_stab_MCG}, the value of the twisted cohomology group $H^{i}(\Gamma_{g,1},F^{\vee}(g))$ for $g\geq N(i,F)$ is isomorphic to $H_{\st}^{i}(F^{\vee})$.

Moreover, the second author proves in \cite{soulie3} a general decomposition for the stable cohomology of the mapping class groups with twisted coefficients given by the dual of a covariant system. For $M\colon \cU\cM_{2} \to \Ab$ covariant system, the colimit $\mathrm{colim}_{g\in(\bN,\leq)}(H_{i}(\Gamma_{g,1};M(g)))$ induced by the canonical injection $\mathfrak{i}_{g} \colon \Gamma_{g,1}\hookrightarrow\Gamma_{g+1,1}$ and the $\Gamma_{g,1}$-equivariant morphism $M([\Sigma_{1,1},\id_{\Sigma_{1+g,1}}])$ is denoted by  $H^{\st}_{i}(M)$; it is called the \emph{stable homology group}. As an application of \cite[Th.~C]{soulie3}, we have:
\begin{thm}\label{thm:stable_functor_homology}
Let $\bK$ be a field such that the stable homology group with constant coefficients $H^{\st}_{i}(\bK)$ is finitely generated as a $\bK$-vector space for each $i\geq 0$. For any covariant system $F\colon \cU\cM_{2} \to \bK\Mod\to \Ab$, we have a natural isomorphism of $\bK$-vector spaces for each $i\geq0$:
$$H_{\st}^{i}(F^{\vee})\cong\underset{k+l=i}{\bigoplus}H_{\st}^{k}(\bK)\underset{\bK}{\otimes} H^{l}(\cU\cM_{2};F).$$
In particular, for $\bK=\bQ$, the stable twisted cohomology module $H_{\st}^{*}(F^{\vee})$ is a free $\Sym_{\bQ}(\cE)$-module.
\end{thm}
\begin{proof}
Since $H_{\st}^{i}(F^{\vee})\cong (H^{\st}_{i}(F))^{\vee}$ by the first isomorphism of \eqref{eq:homology_covariant/cohomology_contravariant}, applying \cite[Th.~C]{soulie3} to $H^{\st}_{i}(F)$ provides that $H_{\st}^{i}(F^{\vee})\cong (\bigoplus_{k+l=i} H^{\st}_{k}(\bK)\underset{\bK}{\otimes} H_{l}(\cU\cM_{2};F))^{\vee}$. Recall that the duality functor preserves finite direct sums and that the $\bK$-vector space $H^{\st}_{i}(\bK)$ is finitely generated, so we deduce from \cite[Chapt.~V, Prop.~4.3]{MacLane_Homology} that $H_{\st}^{i}(F^{\vee})\cong \bigoplus_{k+l=i} (H^{\st}_{k}(\bK))^{\vee}\underset{\bK}{\otimes} (H_{l}(\cU\cM_{2};F))^{\vee}$. Hence the result follows from the second isomorphism of \eqref{eq:homology_covariant/cohomology_contravariant}.
\end{proof}

\begin{rmk}
The result of Theorem~\ref{thm:stable_functor_homology} does not depend on any polynomiality condition, and more generally on whether there is homological stability or not: the formula holds in general for the limit of the cohomology groups which always exists. In particular, this result recovers the previous analogous result due to Djament and Vespa \cite[Prop.~2.22, 2.26]{DV1} when the ambient monoidal structure is \emph{symmetric} monoidal.
\end{rmk}

\paragraph*{Twisted cohomology classes with coefficients in $H$.}

We now review the stable \emph{twisted} cohomology groups with coefficients in the first homology group of the surface.
For the reference \cite{KawazumiMorita}, we will rather quote the preprint version \cite{KawazumiMorita1} as it contains more content and details.

Based on the Harer stability \cite{Harer}, one can prove that $H^{k-1}(\Gamma^{1}_{g}; H(g)) \cong \bigoplus_{i\geq1}H^{k-2i}(\Gamma^{1}_{g}; \bZ)$ and $H^{k-1}(\Gamma_{g,1}; H(g)) \cong \bigoplus_{i\geq1}H^{k-2i}(\Gamma_{g,1}; \bZ)$ in the stable range. This was carried out by Harer \cite[Th.~7.1(b)]{HarerH_3} using $\bQ$ as ground ring for the first time. See \cite[Th.~1.B]{stablecohomologyKawazumi} in the integral case.
More precisely, in \cite{KawazumitwistedMMM}, the first author constructs cohomology classes in $H^{2l-1}(\Gamma^{1}_{g}; H(g))$ for each $l \geq 1$, such that these elements form a basis of the free $H^{*}(\Gamma^{1}_{g}; \bZ)$-module
$H^{*}(\Gamma^{1}_{g}; H(g))$.

The pullbacks on $\Gamma_{g,1}$ of these classes are denoted by $m_{i,1}\in H^{2i-1}(\Gamma_{g,1}; H(g))$ for $i \geq 1$ in the stable range, and provide an isomorphism:
\begin{equation}\label{eq:stablecohomologyKawazumi}
H_{\st}^{*}(H) \cong \bigoplus_{i\geq1}  H_{\st}^{*}(\bZ) m_{i,1}.
\end{equation}
In particular, the stable twisted cohomology $H_{\st}^{*}(H)$ is free as $H_{\st}^{*}(\bZ)$-module. Also, we note that $H_{\st}^{*}(H_{\bQ}(g)) \cong \bigoplus_{i\geq1} \Sym_{\bQ}(\cE) m_{i,1}$ with $\bQ$ as ground ring. See \cite[Th.~1.B]{stablecohomologyKawazumi} for further details. 

We now recollect the definition of the cohomology classes $m_{i,1}$ following their construction of \cite{KawazumiMorita1}, which encompasses the previous related works.
We fix $g\geq2$. Let $p\colon \Gamma^{1}_{g} \twoheadrightarrow \Gamma_{g}$ be the forgetful map of the puncture.
Let $\overline{\Gamma}^{1}_{g}$ be the pullback $\Gamma^{1}_{g}\times_{\Gamma_g} \Gamma^{1}_{g}$. More precisely, there is a defining fibre square
\[
\xymatrix{
\overline{\Gamma}^{1}_{g} \ar[r] \ar[d]_-{\pi}	&  \Gamma^{1}_{g} \ar[d]^-{\overline{p}}\\
\Gamma^{1}_{g} \ar[r]_{p}	 \ar@/_/_{\sigma}[u]						& \Gamma_g,
}
\]
where the section $\sigma: \Gamma^{1}_{g} \to \overline{\Gamma}^{1}_{g}$ is given by $\sigma(\phi) = (\phi, \phi)$. We deduce that there is an isomorphism $\overline{\Gamma}^{1}_{g} \cong   \pi_{1}(\Sigma_{g}) \rtimes\Gamma^{1}_{g}$ defined by $(\phi, \psi) \mapsto (\psi \phi^{-1}, \phi)$.
Under this isomorphism, $\sigma$ is given by $\sigma(\phi) = (1, \phi)$. Similarly to \cite[\S7]{MoritaJacobianII}, this semi-direct product decomposition gives rise to a cocycle $\tilde{k}_{0} \in Z^{1}(\overline{\Gamma}^{1}_{g}; H(g))$, defined by $\tilde{k}_{0}((x, \phi)) = [x]$ for all $x\in \pi_{1}(\Sigma_{g})$ and $\phi\in \Gamma^{1}_{g}$. We denote by $k_{0}$ the element of $H^{1}(\overline{\Gamma}^{1}_{g}; H(g))$ associated to the $1$-cocycle $\tilde{k}_{0}$. Furthermore, we denote by $e\in H^2(\Gamma^{1}_{g};\bZ)$ the Euler class of the short exact sequence \eqref{eq:Birman_SES_Cap} seen as a central extension, and by $\bar{e} \in H^2(\overline{\Gamma}^{1}_{g}; \bZ)$ its pullback induced by the projection map $\overline{\Gamma}^{1}_{g} \to \Gamma^{1}_{g}$, $(\phi, \psi) \mapsto \psi$.

We now consider the Lyndon-Hochschild-Serre spectral sequence associated to the short exact sequence defining $\overline{\Gamma}^{1}_{g} \cong   \pi_{1}(\Sigma_{g}) \rtimes\Gamma^{1}_{g}$. Since $H^{2}(\pi_{1}(\Sigma_{g});\bZ)=\bZ$ while $H^{i}(\pi_{1}(\Sigma_{g});\bZ)=0$ for $i>2$, the projection $\pi\colon \overline{\Gamma}^{1}_{g}\twoheadrightarrow \Gamma^{1}_{g}$ defines a morphism $\pi_{!}\colon H^{*}(\overline{\Gamma}^{1}_{g};H(g))\to H^{*-2}(\Gamma^{1}_{g};(g))$ known as the \emph{Gysin map}.
For all $i\geq1$, we consider the cohomology class in the stable range
\begin{equation}\label{eq:def:mij}
    \pi_{!}(\bar{e}^{i}\cup k_{0}) \in H^{2i-1}(\Gamma^{1}_{g}; H(g)).
\end{equation}
The class $m_{i,1}$ is defined as the preimage in $H_{\st}^{*}(H)$ of the pullback of \eqref{eq:def:mij} to the cohomology group  $H^{2i-1}(\Gamma_{g,1}; H(g))$ along $\mathrm{Cap}:\Gamma_{g,1}\twoheadrightarrow \Gamma^{1}_{g}$ in the stable range.
In particular, using Notation~\ref{not:0_1_trivial_crossed_hom}, Theorem~\ref{thm:Earle_k} and \eqref{eq:stablecohomologyKawazumi}, the class $m_{1,1}$ is the preimage in $H_{\st}^{*}(H)$ of the Earle class $k(g)\in H^{1}(\Gamma_{g,1}; H(g))$ for $g\geq 2$.
Moreover, we know from \cite[Ass.~4.8, (1)]{KawazumitwistedMMM} that the restriction of the class $\pi_{!}(\bar{e}\cup k_{0})$ along the composite $\pi_{1}(UT\Sigma_{g},x) \twoheadrightarrow \pi_{1}(\Sigma_{g},x) \hookrightarrow \Gamma^{1}_{g}$ is equal to $(2-2g)$ times the map $1_{H(g)}\circ \pi_{1}(\varpi) \colon \pi_1(UT\Sigma_{g}) \twoheadrightarrow \pi_1(\Sigma_{g}) \to H(g)$ in $H^{1}(UT\Sigma_{g}; H^{\vee}(g))$. Hence, we deduce from Theorem \ref{thm:Earle_k} that for any $g \geq 2$:
\begin{equation}\label{eq:identification_m_11}
m_{1,1} = k(g).
\end{equation}

\paragraph*{Contraction formula.}
Finally, we recall a classical operation on the twisted Mumford-Morita-Miller classes induced by the contraction of the twisted coefficients.
Let $\mu\colon H^{\vee}(g)\otimes H(g)\to\bZ$ be the intersection pairing associated to Poincar{\'e}-Lefschetz duality. The induced \emph{contraction map} for the cohomology groups is generically denoted by $\mu_{*}$.

\begin{prop}[{\cite[Th.~6.2]{KawazumiMorita1}}]\label{prop:Contraction_Formula}
For all $l,l'\geq1$, for all classes $m_{l,1}\in H_{\st}^{2l-1}(H)$ and $m_{l',1}\in H_{\st}^{2l'-1}(H)$, we have
\begin{equation}\label{eq:contractionformula}
\mu_{*}(m_{l,1},m_{l',1}) = -e_{l+l'-1} \in H_{\st}^{2(l+l'-1)}(\bZ).
\end{equation}
\end{prop}
\begin{proof}[Sketch of proof]
We consider the short exact sequence
\begin{equation}\label{eq:Birman_SES_{2}}
1 \longrightarrow \pi_{1}(\Sigma_{g},x) \longrightarrow \overline{\Gamma}^{1}_{g} \overset{\pi}\longrightarrow \Gamma^{1}_{g} \longrightarrow 1
\end{equation}
induced by pulling back the short exact sequence $1 \longrightarrow \pi_{1}(\Sigma_{g},x) \longrightarrow \Gamma_{g}^{1} \overset{\overline{p}}\longrightarrow \Gamma_{g} \longrightarrow 1$ (i.e.~a copy of \eqref{eq:Birman_SES_1}) along the above fibre square.
As is proved in \cite[Th.~5.3]{KawazumiMorita1}, the Lyndon-Hochschild-Serre spectral sequence for the group extension \eqref{eq:Birman_SES_{2}} induces a canonical decomposition
$$H^{*}(\overline{\Gamma}^{1}_{g}; M) \cong  H^{*}(\Gamma^{1}_{g}; M) \oplus H^{*-1}(\Gamma^{1}_{g}; H(g)\otimes M) \oplus H^{*-2}(\Gamma^{1}_{g}; M)$$
for any $\Gamma^{1}_{g}$-module $M$. 
It is multiplicative and described explicitly by using the cohomology classes $k_{0}$ and $\overline{e}$. The following formula is then deduced from 
a direct computation based on the decomposition: for $M$ and $M'$ two $\overline{\Gamma}^{1}_{g}$-modules, for $m\in H^{*}(\overline{\Gamma}^{1}_{g};M)$ and $m'\in H^{*}(\overline{\Gamma}^{1}_{g};M')$, we have
\begin{equation}\label{eq:contractionformula_general}
(\id_{M}\otimes\mu\otimes \id_{M'})_{*}(\pi_{!}(m\otimes k_{0})\cup\pi_{!}(k_{0}\otimes m')) 
 =-\pi_{!}(m\otimes m') + \sigma^{*}(m)\pi_{!}(m') + \pi_{!}(m)\sigma^{*}(m') - e\pi_{!}(m)\pi_{!}(m').
\end{equation}
Here, we recall that $\sigma: \Gamma^{1}_{g} \to \overline{\Gamma}^{1}_g$, $\phi \mapsto (\phi, \phi)$, is the diagonal map. 
In particular, we have $\mu_{*}(m_{l,1},m_{l',1}) = -e_{l+l'-1} + e^{l}e_{l'-1} + e^{l'}e_{l-1} - ee_{l-1}e_{l'-1}$ in the stable range. Since the Euler class $e$ vanishes on $\Gamma_{g,1}$, we deduce formula \eqref{eq:contractionformula} from \eqref{eq:contractionformula_general}.
\end{proof}

\begin{rmk}
Analogous formulas to \eqref{eq:contractionformula_general} are computed in \cite[Prop.~3.10]{KupersRandal-Williams} for mapping class groups of surfaces and higher even-dimensional manifolds. There is no $-1$ sign for the analogue to formula \eqref{eq:contractionformula} in \cite[\S3]{KupersRandal-Williams} because of a difference of conventions with \cite{KawazumiMorita1} in the identification of $H^{1}(\Sigma_{g,1}; \bZ)$ with $H(g)$.
\end{rmk}

\subsubsection{Exotic situation}\label{s:exotic}
On the basis of current knowledge, contrary to the covariant cases, there is no general framework for cohomological stability with twisted coefficients given by \emph{covariant} systems. In particular, there is no result analogous to Proposition~\ref{prop:hom_stab_MCG} replacing $F^{\vee}$ in that statement by the functor $\tilde{H}\colon\cU\cM_{2} \to \Ab$ defined by the groups $\{H^{1}(UT\Sigma_{g,1};\bZ),g\in\bN\}$.
However, the functor $\tilde{H}$ does satisfy homological stability as follows.

\begin{prop}\label{prop:homological_stability_covariant}
For each $g\geq0$, the image by the functor $\tilde{H}$ of the morphism of $\cU\cM_{2}$ of type $[\Sigma_{1,1},\id_{\Sigma_{g+1,1}}]$ has a canonical $\Gamma_{g,1}$-equivariant splitting denoted by $\tilde{H}^{-1}([\Sigma_{1,1},\id_{\Sigma_{g+1,1}}])$.

Then, for each $i\geq0$, let $\Psi_{i,g}$ denote the canonical map $ H^{i}(\Gamma_{g+1,1}; \tilde{H}(g+1)) \to H^{i}(\Gamma_{g,1}; \tilde{H}(g))$ induced by the canonical canonical injection $\mathfrak{i}_{g} \colon \Gamma_{g,1}\hookrightarrow\Gamma_{g+1,1}$ and the splitting $\tilde{H}^{-1}([\Sigma_{1,1},\id_{\Sigma_{g+1,1}}])$. If $g\geq2i+5$, then $\Psi_{i,g}$ is an isomorphism. 
\end{prop}
\begin{proof}
First of all, since $H$ is a strong monoidal functor $(\cU\cM_{2},\natural,\bD^{2}) \to (\Ab,\oplus,0)$ (see Example~\ref{eg:Z_H_poly_functor}), the morphism $H([\Sigma_{1,1},\id_{\Sigma_{g+1,1}}])$ has a canonical $\Gamma_{g,1}$-equivariant splitting, that we denote by $H^{-1}([\Sigma_{1,1},\id_{\Sigma_{g+1,1}}])$. Now, we fix the  framings $\xi$ and $\xi'$ on $\Sigma_{g,1}$ and $\Sigma_{1,1}$, and we view $\tilde{H}(g+1)$ as a $\Gamma_{g+1,1}$-representation via the canonical injection $\mathfrak{i}_{g} \colon \Gamma_{g,1}\hookrightarrow\Gamma_{g+1,1}$. Then, it follows from the representation structure of $\tilde{H}(g+1)$ described by \eqref{eq:matrix_action}, from the fact that $\Gamma_{g,1}$ acts trivially on the subsurface $\Sigma_{1,1}\hookrightarrow \Sigma_{1,1}\natural \Sigma_{g,1} \simeq \Sigma_{g+1,1}$, from Lemma~\ref{lem:glued_crossed_hom} and from the fact that $H$ is a strong monoidal functor that there is a $\Gamma_{g,1}$-equivariant decomposition $\tilde{H}(g+1)\cong H(1) \oplus \tilde{H}(g)$. Therefore, the induced projection $\tilde{H}(g+1) \twoheadrightarrow \tilde{H}(g)$ is the canonical $\Gamma_{g,1}$-equivariant splitting $\tilde{H}^{-1}([\Sigma_{1,1},\id_{\Sigma_{g+1,1}}])$ of $\tilde{H}([\Sigma_{1,1},\id_{\Sigma_{g+1,1}}])$, and we construct from \eqref{SEShutb} the following commutative diagram where the rows are short exact sequences:
\begin{equation}\label{eq:proof_stability_tilde_H}
\centering
\begin{split}
\xymatrix{0 \ar@{->}[r]
& \bZ \ar@{->}[r]  \ar@{=}[d]
& \tilde{H}(g+1)\ar@{->}[d]^-{\tilde{H}^{-1}([\Sigma_{1,1},\id_{\Sigma_{g+1,1}}])} \ar@{->}[rr]
&& H(g+1) \ar@{->}[d]^-{H^{-1}([\Sigma_{1,1},\id_{\Sigma_{g+1,1}}])} \ar@{->}[r]
& 0\\
0 \ar@{->}[r]
& \bZ\ar@{->}[r]
& \tilde{H}(g)\ar@{->}[rr]
&& H(g) \ar@{->}[r]
& 0.}
\end{split}
\end{equation}
Considering the long exact sequences in cohomology associated to the two rows of \eqref{eq:proof_stability_tilde_H}, we obtain the following commutative diagram for each $g\geq0$, where the rows are exact sequences and the vertical arrows $\Phi_{i,g}(\bZ)$ and $\Phi_{i,g}(H)$ are both isomorphisms when $g\geq 2i+5$:
$$\xymatrix{\cdots\ar@{->}[r] & 
H^{i}(\Gamma_{g+1,1};\bZ) \ar@{->}[r]  \ar@{->}[d]^-{\Phi_{i,g}(\bZ)}_{\cong}
& H^{i}(\Gamma_{g+1,1};\tilde{H}^{\vee}(g+1))\ar@{->}[d]^-{\Psi_{i,g}} \ar@{->}[r]
& H^{i}(\Gamma_{g+1,1};H(g+1)) \ar@{->}[d]^-{\Phi_{i,g}(H)}_{\cong} \ar@{->}[r]
& \cdots \\
\cdots\ar@{->}[r]
& H^{i}(\Gamma_{g,1};\bZ)\ar@{->}[r]
& H^{i}(\Gamma_{g,1};\tilde{H}^{\vee}(g))\ar@{->}[r]
& H^{i}(\Gamma_{g,1};H(g)) \ar[r]
& \cdots}$$
The results thus follows from a clear recursion on $i$ and the five lemma.
\end{proof}
\begin{defn}\label{def:stable_cohomology_covariant}
For each $i\geq0$, the \emph{stable twisted cohomology group} $H_{\st}^{i}(\tilde{H})$ is the inverse limit $\underleftarrow{\lim}_{g\geq0}H^{i}(\Gamma_{g,1}; \tilde{H}(g))$ induced by the maps $\{\Psi_{i,g}\}_{g\in\bN}$. Moreover, there is a canonical $H_{\st}^{*}(\bZ)$-module structure on the cohomology groups $H_{\st}^{*}(\tilde{H}):=\bigoplus_{i\geq0}H_{\st}^{i}(\tilde{H})$, induced by the cup product and the compatibility of the stabilisation maps $\{\Phi'_{i_{1},g}(\bZ),\Psi_{i_{2},g}\}_{i_{1}, i_{2}, g\in\bN}$ with respect to that operation.
\end{defn}

\section{Stable twisted cohomology computations}\label{sec:first_homology_group_systems}

In this section, we prove Theorems~\ref{thm:main_thm_0} and \ref{thm:main_thm_1}; see Theorems~\ref{thm:stable_cohomology_contravariant_tildeH} and \ref{thm:explicit-generators-stable-cohomology-tildeH}. In both cases, the work relies on the determination of the connecting morphisms of the long exact sequences associated to the short exact sequences \eqref{SESchutb} and \eqref{SEShutb} respectively.
As a preliminary, we recall the following key result, which expresses the connecting homomorphism of a cohomology long exact sequence as the composition product (also known as the Yoneda product) with an extension class; we refer to \cite[\S7]{BourbakiAlgHomol} for further details about this notion.
\begin{prop}[{\cite[\S7, Prop.~5.a)]{BourbakiAlgHomol}}]\label{prop:connecting_hom_cup_product}
Let $G$ be a group and $R$ be a principal ideal domain. We consider a short exact sequence $(S):0\to K \to M \to C \to 0$ of left $R[G]$-module, classified by $\kappa\in \Ext^{1}_{R[G]}(C,K)$.
For each $i\geq0$, the connecting homomorphism $\delta^{i}\colon H^{i}(G;C)\to H^{i+1}(G; K)$ of the cohomology long exact sequence associated with $(S)$ is equal to the composition product $c\mapsto \kappa\circ c$ by $\kappa$, denoted by $\kappa\circ -$.
\end{prop}
Furthermore, we fix the following conventions for notations for all the remainder of \S\ref{sec:first_homology_group_systems}.
\begin{convention}\label{conv:section_4}
From now on, we implicitly assume that $g\geq2i+5$ each time we consider a cohomological degree $i$ for $H^{i}(\Gamma_{g,1};M(g))$ where $M(g)=\tilde{H}(g)$ or $\tilde{H}^{\vee}(g)$, for the homological stability bound of Proposition~\ref{prop:hom_stab_MCG} to be reached.
Also, we denote by $H_{\st}^{\odd}(M)$ and $H_{\st}^{\even}(M)$ the $\bN$-graded submodules of $H_{\st}^{*}(M)$ defined by $\{H_{\st}^{\mathrm{2i+1}}(M),i\in\bN\}$ and $\{H_{\st}^{\mathrm{2i}}(M),i\in\bN\}$ respectively.
\end{convention}

\subsection{First cohomology group system}\label{subsec:first_cohomology_group_system}

We start by studying the stable cohomology groups of the mapping class groups $\Gamma_{g,1}$ with twisted coefficient given by $\tilde{H}^{\vee}(g)$.
We recall from \S\ref{sec:unit_tangent_bundle_representations} that $k(g)$ is the extension class of the short exact sequence \eqref{SESchutb} in the stable range. We note that the composition product $k(g)\circ - \colon H_{\st}^{i}(\bZ)\to H_{\st}^{i+1}(H)$ coincides with the cup product $k(g)\cup -$ by $k(g)$; see \cite[Chapter~5, Th.~4.6]{Brown} for instance. Since $k(g)=m_{1,1}$ for $g\geq2$ (see \eqref{eq:identification_m_11}), it follows from Proposition~\ref{prop:connecting_hom_cup_product} that the cohomology long exact sequence applied to the rational version of \eqref{SESchutb} may be written as follows in the stable range (i.e.~$g\geq2i+5$):
\begin{equation}\label{eq:LES_contravariant_tildeH}
\xymatrix{\cdots \ar@{->}[r] & H_{\st}^{i}(\tilde{H}^{\vee})\ar@{->}[r]^-{H_{\st}^{i}(\iota^{*})} & H_{\st}^{i}(\bZ)\ar@{->}[rr]^-{m_{1,1}\cup -} && H_{\st}^{i+1}(H)\ar@{->}[rr]^-{H_{\st}^{i+1}(\varpi^{*})} && H_{\st}^{i+1}(\tilde{H}^{\vee}) \ar@{->}[r]  & \cdots.}
\end{equation}
For each $i\geq0$, we denote the class $H_{\st}^{2i-1}(\varpi^{*})(m_{i,1})$ in $H_{\st}^{2i-1}(\tilde{H}^{\vee})$ by $\tilde{m}_{i,1}$. We deduce that:
\begin{thm}\label{thm:stable_cohomology_contravariant_tildeH}
There is a $H_{\st}^{*}(\bZ)$-module isomorphism $H_{\st}^{*}(\tilde{H}^{\vee})\cong\bigoplus_{i\geq2} H_{\st}^{*}(\bZ)\tilde{m}_{i,1}$.
\end{thm}
\begin{proof}
We recall from \eqref{eq:stablecohomologyKawazumi} that the stable cohomology module $H_{\st}^{*}(H)$ is isomorphic to the free $H_{\st}^{*}(\bZ)$-module with basis $\{m_{i,1},i\geq1\}$. Hence, the map $m_{1,1}\cup -$ being defined by $m_{1,1}\cup e_{\alpha}=e_{\alpha} m_{1,1}$ for all $\alpha\geq1$, it induces via the long exact sequence \eqref{eq:LES_contravariant_tildeH} an injective $H_{\st}^{*}(\bZ)$-module morphism $m_{1,1}\cup -\colon H_{\st}^{*}(\bZ)\hookrightarrow \bigoplus_{i\geq1} H_{\st}^{*}(\bZ)m_{i,1}$ which cokernel is isomorphic to $H_{\st}^{*}(\tilde{H}^{\vee})$. The result thus follows from the obvious direct computation of this cokernel.
\end{proof}

Over the rationals, since $H_{\st}^{\odd}(\bQ) = 0$ by \eqref{eq:Madsen_weiss_thm}, we deduce from Theorem~\ref{thm:stable_cohomology_contravariant_tildeH} that:
\begin{coro}\label{coro:stable_cohomology_contravariant_tildeH}
The $\Sym_{\bQ}(\cE)$-module $H_{\st}^{\odd}(\tilde{H}^{\vee}_{\bQ})$ is isomorphic to $\bigoplus_{i\geq2} \Sym_{\bQ}(\cE)\tilde{m}_{i,1}$ while the $\Sym_{\bQ}(\cE)$-module $H_{\st}^{\even}(\tilde{H}^{\vee}_{\bQ})$ is null.
\end{coro}

\begin{rmk}[\emph{Interpretation in terms of functor homology.}]\label{rmk:functor_homology_interpretation}
The modules $\tilde{H}^{\vee}_{\bQ}$ define a \emph{contravariant} twisted coefficient system; see Proposition~\ref{prop:tilde_H_functor}. It follows from Theorem~\ref{thm:stable_functor_homology} that $H_{\st}^{*}(\tilde{H}^{\vee}_{\bQ})\cong \Sym_{\bQ}(\cE)\otimes_{\bQ}H^{*}(\cU\cM_{2};\tilde{H}_{\bQ})$.
In particular, this explains why the stable cohomology is a free $\Sym_{\bQ}(\cE)$-module.
Then the long exact sequence for the homology of categories (see \cite[\S2]{FranjouPira} for instance) associated with the short exact sequence \eqref{SESchutb} directly gives that $H^{i}(\cU\cM_{2};\tilde{H}_{\bQ})\cong H^{i}(\cU\cM_{2};H_{\bQ})$ if $i\geq2$.
However, we need the above reasoning to compute that $H^{0}(\cU\cM_{2};\tilde{H}_{\bQ})= H^{1}(\cU\cM_{2};\tilde{H}_{\bQ}) =0$.
\end{rmk}

\subsection{First homology group system}\label{subsec:first_homology_group_system}

Contrasting with \S\ref{subsec:first_cohomology_group_system}, we study here the stable cohomology groups of the mapping class groups with twisted coefficient in the first homology group of the unit tangent bundle of the surface.

\subsubsection{Determination of the connecting homomorphism}\label{sss:determination_connecting_hom}

For each $i\geq0$, let $\delta_{\eqref{SEShutb}}^{i}\colon H^{i}(\Gamma_{g,1};H(g))\to H^{i+1}(\Gamma_{g,1}; \bZ)$ be the $i$th connecting homomorphism of the cohomology long exact sequence associated with the extension \eqref{SEShutb}. Our aim in this section is to determine $\delta_{\eqref{SEShutb}}^{i}$ in terms of simpler operations, that we may handle for stable twisted cohomology computations in \S\ref{sss:computation_stable_twisted}; see Proposition~\ref{prop:connecting_hom_determination_cov}.
Firstly, since $H(g)$ is a torsion-free abelian group, tensoring (on the left) the short exact sequence \eqref{SESchutb} with $H(g)$ provides an extension:
\begin{equation}\label{eq:tensor_SESchutb}
\xymatrix{0\ar@{->}[r] &  H^{\vee}(g) \otimes H(g)\ar@{->}[rr]^-{\varpi^{*}\otimes \id_{H(g)}} && \tilde{H}^{\vee}(g)\otimes H(g)\ar@{->}[rr]^-{\iota^{*}\otimes \id_{H(g)}} && H(g)\ar@{->}[r] & 0.}
\end{equation}
Let $\delta_{\eqref{eq:tensor_SESchutb}}^{i}$ denote the $i$th connecting homomorphism of the cohomology long exact sequence associated with \eqref{eq:tensor_SESchutb}.
\begin{lem}\label{lem:connecting_hom_tensor}
For each $g\geq0$, the connecting homomorphism $\delta_{\eqref{eq:tensor_SESchutb}}^{i}$ is equal to $k(g)\cup -$.
\end{lem}
\begin{proof}
Let $\delta_{\eqref{SESchutb}}^{i}$ denote the $i$th connecting homomorphism of the cohomology long exact sequence associated \eqref{SESchutb}. We recall from \S\ref{subsec:first_cohomology_group_system} that $\delta_{\eqref{SESchutb}}^{i}=k(g)\cup -$ in the stable range, and so $\delta_{\eqref{SESchutb}}^{1}(1)=k(g)$ for $1\in H^{0}(\Gamma_{g,1};\bZ)$. Then, it follows from \cite[Chapter~5, (3.3)]{Brown} that $\delta_{\eqref{eq:tensor_SESchutb}}^{i}(v)=\delta_{\eqref{eq:tensor_SESchutb}}^{i}(1\cup v)= \delta_{\eqref{SESchutb}}^{1}(1)\cup v = k(g)\cup v$ for all $v\in H^{i}(\Gamma_{g,1};H(g))$.
\end{proof}
Now, we recall from \S\ref{ss:covariant_systems_TSH} that $\mu\colon H^{\vee}(g)\otimes H(g)\to\bZ$ denotes the intersection pairing associated to Poincar{\'e}-Lefschetz duality; in particular, it defines a class $\mu\in \Ext^{0}_{\bZ[\Gamma_{g,1}]}(H^{\vee}(g)\otimes H(g),\bZ)$.
Following Proposition~\ref{prop:Contraction_Formula}, we denote by $\mu_{i}(k(g),-)$ the map $H^{i}(\Gamma_{g,1};H(g))\to H^{i+1}(\Gamma_{g,1};\bZ)$ defined by $v \mapsto \mu_{*}(k(g), v)$ induced by the contraction map $\mu_{*}$.
\begin{prop}\label{prop:connecting_hom_determination_cov}
For $g\geq0$, the connecting homomorphism $\delta_{\eqref{SEShutb}}^{i}$ is equal to $\mu_{i}(k(g),-)$.
\end{prop}
\begin{proof}
Let $\Upsilon\colon \tilde{H}^{\vee}(g)\otimes H(g) \to \tilde{H}(g)$ be the abelian group morphism defined by $(f\otimes x) \mapsto (f(0,x),x)$. Also, we have $\Upsilon(\phi\cdot (f, x))= (f(\phi^{-1}(0,\phi\cdot x)),\phi\cdot x) = (f(0,x),\phi\cdot x) =\phi\cdot \Upsilon(f, x)$ for all $\phi\in\Gamma_{g,1}$, so the map $\Upsilon$ is $\Gamma_{g,1}$-equivariant. Now, we consider the following diagram in the category of $\bZ[\Gamma_{g,1}]$-modules, where the top row is \eqref{eq:tensor_SESchutb} and the bottom row is \eqref{SEShutb}:
\begin{equation}\label{eq:proof_connectiong_hom_tilde_H}
\centering
\begin{split}
\xymatrix{0\ar@{->}[r] & 
H^{\vee}(g) \otimes H(g) \ar@{->}[rr]^-{\varpi^{*}\otimes \id_{H(g)}}  \ar@{->}[d]^-{\mu}
&& \tilde{H}^{\vee}(g) \otimes H(g)\ar@{->}[d]^-{\Upsilon} \ar@{->}[rr]^-{\iota^{*}\otimes \id_{H(g)}}
&& H(g) \ar@{=}[d]^-{\id_{H(g)}} \ar@{->}[r]
& 0 \\
0\ar@{->}[r]
& \bZ\ar@{->}[rr]^-{\iota_{*}}
&& \tilde{H}(g)\ar@{->}[rr]^-{\varpi_{*}}
&& H(g) \ar[r]
& 0.}
\end{split}
\end{equation}
For all $f \in \tilde{H}^{\vee}(g)$ and $x\in H(g)$, we have $(\varpi_{*}\circ \Upsilon)(f\otimes x)= x =(\iota^{*}\otimes \id_{H(g)})(f \otimes x)$, which proves that the right-hand square of \eqref{eq:proof_connectiong_hom_tilde_H} is commutative.
Furthermore, the isomorphism \eqref{eq:dual_H} induced by the Poincar{\'e}-Lefschetz duality is explicitly given by intersection pairing via the map $H(g)\overset{\sim}{\rightarrow} H^{\vee}(g)$ defined by $y\mapsto \mu(y,-)$. We compute that $(\Upsilon\circ(\varpi^{*}\otimes \id_{H(g)}))(\mu(y,-)\otimes x) = (\mu(y,x), 0) + (0,x)$ for all $x,y\in H(g)$.
For $(0,x)\in \Image(\Upsilon\circ(\varpi^{*}\otimes \id_{H(g)}))$, note that $\varpi^{*}(0,x)=0$ since the right-hand square of \eqref{eq:proof_connectiong_hom_tilde_H} is commutative, and so there exists $z\in\bZ$ such that $\iota_{*}(z)=(z,0)=(0,x)$ because the bottom row is exact. Hence the element $(0,x)$ is null in the image of $\Upsilon\circ(\varpi^{*}\otimes \id_{H(g)})$. We deduce that $(\Upsilon\circ(\varpi^{*}\otimes \id_{H(g)}))(\mu(y,-)\otimes x)=(\iota_{*}\circ \mu)(y\otimes x)$, so the left-hand square of \eqref{eq:proof_connectiong_hom_tilde_H} is commutative, which proves that the full diagram \eqref{eq:proof_connectiong_hom_tilde_H} is commutative.

We recall from \S\ref{sec:unit_tangent_bundle_representations} that $k^{\vee}(g)$ is the extension class of \eqref{SEShutb}. Also, the extension class of \eqref{eq:tensor_SESchutb} in $\Ext^{1}_{\bZ[\Gamma_{g,1}]}(H^{\vee}(g)\otimes H(g),H(g))$ is the tensor product $k(g)\otimes \id_{H(g)}$ of the trivial class $\id_{H(g)}\in \Ext^{0}_{\bZ[\Gamma_{g,1}]}(H(g);H(g))$ with the class $k(g)\in\Ext^{1}_{\bZ[\Gamma_{g,1}]}(\bZ;H(g))$ of the extension \eqref{SESchutb}.  Therefore, it follows from \cite[\S7, Prop.~4]{BourbakiAlgHomol} and from the commutativity of \eqref{eq:proof_connectiong_hom_tilde_H} that the class $k^{\vee}(g)$ is equal to the composition product $\mu \circ (k(g)\otimes \id_{H(g)})$ in $\Ext^{1}_{\bZ[\Gamma_{g,1}]}(H(g),\bZ)$.
Then, we deduce from Proposition~\ref{prop:connecting_hom_cup_product} and from the associativity of the composition product (see \cite[\S7.1]{BourbakiAlgHomol}) that $\delta_{\eqref{SEShutb}}^{i}=\mu \circ ((k(g)\otimes \id_{H(g)})\circ-)=\mu \circ \delta_{\eqref{eq:tensor_SESchutb}}^{i}$. Since the contraction map $\mu_{*}$ is by definition the composition product with $\mu$, the result follows from Lemma~\ref{lem:connecting_hom_tensor}.
\end{proof}

\subsubsection{Computation of the stable twisted cohomology}\label{sss:computation_stable_twisted}
In the stable range (i.e.~$g\geq2i+5$), since $k(g)=m_{1,1}$ for $g\geq2$ (see \eqref{eq:identification_m_11}), it follows from Proposition~\ref{prop:connecting_hom_determination_cov} that the cohomology long exact sequence applied to \eqref{SESchutb} may be written as follows:
\begin{equation}\label{eq:LES_tilde_H}
\xymatrix{\cdots \ar@{->}[r] & H_{\st}^{2i+1}(\tilde{H})\ar@{->}[r] & H_{\st}^{2i+1}(H)\ar@{->}[rr]^-{\mu_{i}(m_{1,1},-)} &  & H_{\st}^{2i+2}(\bZ)\ar@{->}[r] & H_{\st}^{2i+2}(\tilde{H})\ar@{->}[r] & \cdots}
\end{equation}
By the contraction formula \eqref{eq:contractionformula}, we compute that $\mu_{i}(m_{1,1}, m_{\alpha,1})=-e_{\alpha}$ for all $\alpha\geq1$.
Recall from \eqref{eq:stablecohomologyKawazumi} that $H_{\st}^{*}(H) \cong \bigoplus_{i\geq1}  H_{\st}^{*}(\bZ) [2i-1]$. Let $\Xi\colon H_{\st}^{\odd}(\bZ) \to H_{\st}^{*}(\bZ)$ be the $H_{\st}^{*}(\bZ)$-module map induced by the multiplication by $e_{i}$ on each summand $H_{\st}^{*}(\bZ) [2i-1]$; its cokernel is isomorphic to $H_{\st}^{*}(\tilde{H})\otimes_{\bZ[\cE]}\bZ$. We deduce from the exactness of \eqref{eq:LES_tilde_H} that:
\begin{prop}\label{prop:stable_cohomology_first_homology_group_result_integral}
The graded $H_{\st}^{*}(\bZ)$-module $H_{\st}^{\even}(\tilde{H})$ is isomorphic to an extension of $\ker(\Xi)$ by $H_{\st}^{*}(\bZ)\otimes_{\bZ[\cE]}\bZ$.
\end{prop}
To make further computations, we must restrict to considering the rational homology representations.
First of all, since $\Sym_{\bQ}(\cE)=H_{\st}^{*}(\bQ)$ is concentrated in even degrees (see \eqref{eq:Madsen_weiss_thm}), the graded cohomology groups $H_{\st}^{\odd}(\tilde{H}_{\bQ})$ and $H_{\st}^{\even}(\tilde{H}_{\bQ})$ inherit canonical $\Sym_{\bQ}(\cE)$-module structures from that of $H_{\st}^{*}(\tilde{H}_{\bQ})$ induced by the cup product, and then the decomposition
\begin{equation}\label{eq:decomposition_parity_twisted_cohomology}
H_{\st}^{*}(\tilde{H}_{\bQ}) = H_{\st}^{\even}(\tilde{H}_{\bQ})\oplus H_{\st}^{\odd}(\tilde{H}_{\bQ})
\end{equation}
is stable under the action of the algebra $\Sym_{\bQ}(\cE)$.
Furthermore, since $H_{\st}^{\odd}(\bQ) = 0$ by \eqref{eq:Madsen_weiss_thm} while $H_{\st}^{\even}(H_{\bQ}) = 0$ by \eqref{eq:stablecohomologyKawazumi}, we deduce that, for each $i\geq0$, the map $\mu_{2i+1}(m_{1,1},-)\colon H_{\st}^{2i+1}(H_{\bQ})\to H_{\st}^{2i+2}(\bQ)$ is surjective while $\mu_{2i}(m_{1,1},-)\colon H_{\st}^{2i}(H_{\bQ})\to H_{\st}^{2i+1}(\bQ)$ is null. We denote by $\theta$ the stable $0$th-cohomology class defined by the fibre of the locally trivial fibration $\bS^{1}\overset{\iota}{\hookrightarrow}UT\Sigma_{g,1}\overset{\varpi}{\to}\Sigma_{g,1}$; in particular, we have $H_{\st}^{0}(\tilde{H}_{\bQ})\cong\bQ\theta$.
Then we deduce from the rational version of cohomology long exact sequence \eqref{eq:LES_tilde_H} that:
\begin{coro}\label{cor:stable_cohomology_first_homology_group_result}
The $\Sym_{\bQ}(\cE)$-module $H_{\st}^{\even}(\tilde{H}_{\bQ})$ is isomorphic to the trivial $\Sym_{\bQ}(\cE)$-module $H_{\st}^{0}(\tilde{H}_{\bQ})\cong\bQ\theta$ (i.e.~each class $e_{i}$ acts as zero on $\bQ\theta$) and $H_{\st}^{\odd}(\tilde{H}_{\bQ})$ is isomorphic to the kernel of the graded morphism $\mu_{\odd}(m_{1,1},-)=\bigoplus_{i\geq0} \mu_{2i+1}(m_{1,1},-)\colon H_{\st}^{\odd}(H_{\bQ})\to \Sym_{\bQ}(\cE)$.
\end{coro}
Therefore, we have the following exact sequence of $\Sym_{\bQ}(\cE)$-modules:
\begin{equation}\label{ESasmodule}
\xymatrix{H_{\st}^{\odd}(\tilde{H}_{\bQ})\ar@{^{(}->}[r] & H_{\st}^{\odd}(H_{\bQ})\ar@{->}[rr]^-{\mu_{\odd}(m_{1,1},-)} &  & \Sym_{\bQ}(\cE)\ar@{->}[r]^-{\aug} & \bQ\ar@{->}[r] & 0.}
\end{equation}

\paragraph*{Computations of $\Tor$-groups.}
In order to give some qualitative properties of the stable twisted cohomology groups that we study, we compute the $\Tor$-groups of the stable twisted cohomology groups $H_{\st}^{*}(\tilde{H}_{\bQ})$ as $\Sym_{\bQ}(\cE)$-module.
\begin{thm}\label{thm:homology_algebra_first}
For any $j \geq 0$, we have 
$\Tor_{j}^{\Sym_{\bQ}(\cE)}(\bQ,H_{\st}^{*}(\tilde{H}_{\bQ}))\cong\Lambda^{j}\cE\oplus\Lambda^{j+2}\cE$. In particular, the $\Sym_{\bQ}(\cE)$-module $H_{\st}^{*}(\tilde{H}_{\bQ})$ is not free and $\Tor_{0}^{\Sym_{\bQ}(\cE)}(\bQ,H_{\st}^{*}(\tilde{H}_{\bQ}))\cong\Lambda^{2}\cE\oplus \bQ$.
\end{thm}

\begin{proof}
Using the splitting \eqref{eq:decomposition_parity_twisted_cohomology}, the results follow from the respective computations of the left derived functors $\Tor_{*}^{\Sym_{\bQ}(\cE)}(\bQ,-)$ of the $\Sym_{\bQ}(\cE)$-modules $H_{\st}^{\even}(\tilde{H}_{\bQ})$ and $H_{\st}^{\odd}(\tilde{H}_{\bQ})$. Since $\Sym_{\bQ}(\cE)$ is a polynomial algebra on the vector space $\cE$, we deduce that $\Tor_{j}^{\Sym_{\bQ}(\cE)}(\bQ,\bQ)\cong\Lambda^{j}\cE$ by using the resolution given by the Koszul complex (see \cite[\S3.4.6]{Loday} or \cite[\S9.1.]{BourbakiAlgHomol} for instance). Hence $\Tor_{j}^{\Sym_{\bQ}(\cE)}(\bQ,H_{\st}^{\even}(\tilde{H}_{\bQ}))\cong \Lambda^{j}\cE$ because $H_{\st}^{\even}(\tilde{H}_{\bQ})\cong \bQ\theta \cong \bQ$ as a $\Sym_{\bQ}(\cE)$-module by Corollary~\ref{cor:stable_cohomology_first_homology_group_result}.

Now, we recall that $H_{\st}^{*}(H_{\bQ})=H_{\st}^{\odd}(H_{\bQ})$ is a free $\Sym_{\bQ}(\cE)$-module by \eqref{eq:stablecohomologyKawazumi}. Then we note that the exact sequence of $\Sym_{\bQ}(\cE)$-modules \eqref{ESasmodule} has its two interior terms which are free (and thus projective) $\Sym_{\bQ}(\cE)$-modules. Hence a projective $\Sym_{\bQ}(\cE)$-module resolution of $H_{\st}^{\odd}(\tilde{H}_{\bQ})$ induces via \eqref{ESasmodule} a projective $\Sym_{\bQ}(\cE)$-module resolution of $H_{\st}^{\odd}(\tilde{H}_{\bQ})$. Therefore $\Tor_{j}^{\Sym_{\bQ}(\cE)}(\bQ,H_{\st}^{\odd}(\tilde{H}_{\bQ}))\cong\Tor_{j+2}^{\Sym_{\bQ}(\cE)}(\bQ,\bQ)\cong\Lambda^{j+2}\cE$ for all $j\geq0$ thanks to the Koszul resolution, which ends the proof.
\end{proof}

In the paper \cite{KawazumiSoulieII}, we will compute $\Tor_{j}^{\Sym_{\bQ}(\cE)}(\bQ,H_{\st}^{*}(\Lambda^{d}\tilde{H}_{\bQ}))$
for $2 \leq d \leq 5$. In particular, this group is non-trivial for each degree $j \geq 0$, except for the case $d=2$.

\paragraph*{Computation of the stable cohomology.}
Finally, we explicitly describe the generators and relations of the module $H_{\st}^{\odd}(\tilde{H}_{\bQ})$ as follows.
For each pair $(i,j)$ of non-null natural numbers such that $j>i$, note that $\mu_{\odd}(m_{1,1},e_{i}m_{j,1}- e_{j}m_{i,1})=0$, so there is a unique non-trivial class in $H^{2(i+j)-1}_{\st}(\tilde{H}_{\bQ})$ which is mapped to $e_{i}m_{j,1}- e_{j}m_{i,1}$ along $H^{2(i+j)-1}_{\st}(\tilde{H}_{\bQ})\hookrightarrow H^{2(i+j)-1}_{\st}(H_{\bQ})$, that we also denote by $e_{i}m_{j,1}- e_{j}m_{i,1}$. Let $\mathfrak{M}$ be the quotient of the free $\Sym_{\bQ}(\cE)$-module generated by the  $\{M_{i,j};j>i\geq1\}$ by the $\Sym_{\bQ}(\cE)$-submodule generated by the elements $\{e_i M_{j,k} - e_j M_{i,k} + e_k M_{i,j};k > j > i \geq 1\}$.
Using this description, we prove the following.
\begin{thm}\label{thm:explicit-generators-stable-cohomology-tildeH}
Mapping $M_{i,j}$ to $e_{i} m_{j,1} - e_{j} m_{i,1}$ for each $j>i\geq1$ defines a $\Sym_{\bQ}(\cE)$-module isomorphism from $\mathfrak{M}$ to $H_{\st}^{\odd}(\tilde{H}_{\bQ})$.
\end{thm}
\begin{proof}
We consider the $\Sym_{\bQ}(\cE)$-modules $\Omega_{\Sym_{\bQ}(\cE)|\bQ}^{n}:=\Sym_{\bQ}(\cE)\otimes\Lambda^{n}\cE$ for all $n\geq0$.
We recall from \cite[\S3]{BourbakiAlgHomol} that the Koszul complex associated to the $\Sym_{\bQ}(\cE)$-modules $\{\Omega_{\Sym_{\bQ}(\cE)|\bQ}^{n}\}_{n\in\bN}$ defines a free resolution of $\bQ$ as a $\Sym_{\bQ}(\cE)$-module. More precisely, there is an exact sequence of $\Sym_{\bQ}(\cE)$-modules
\begin{equation}\label{eq:differential_forms_resolution}
\xymatrix{\cdots\ar@{->}[r]^-{\partial_{n+1}} & \Omega_{\Sym_{\bQ}(\cE)|\bQ}^{n}\ar@{->}[r]^-{\partial_{n}} & \Omega_{\Sym_{\bQ}(\cE)|\bQ}^{n-1}\ar@{->}[r]^-{\partial_{n-1}} & \cdots\ar@{->}[r]^-{\partial_{1}} & \Sym_{\bQ}(\cE)\ar@{->}[r] & \bQ\ar@{->}[r] & 0,}
\end{equation}
where the differential $\partial_{n}: \Omega_{\Sym_{\bQ}(\cE)|\bQ}^{n} \to \Omega^{n-1}_{\Sym_{\bQ}(\cE)|\bQ}$ is the $\Sym_{\bQ}(\cE)$-module morphism defined by
\begin{equation}\label{eq:def_P_D_n}
\partial_{n}(x \otimes (e_{j_{1}}\wedge\cdots \wedge e_{j_n})) = \sum^{n}_{i=1}(-1)^{i+1} (e_{j_{i}}\cdot x) \otimes (e_{j_{1}}\wedge\cdots\wedge \widehat{e_{j_{i}}} \wedge\cdots \wedge e_{j_n})
\end{equation}
for each $x \in \Sym_{\bQ}(\cE)$, where $e_{j_{i}}\cdot x$ is the product of $e_{j_{i}}$ and $x$ in the algebra $\Sym_{\bQ}(\cE)$.
From now on, we identify the twisted Mumford-Morita-Miller class $m_{i,1}$ with the element $1\otimes e_{i}\in\Omega_{\Sym_{\bQ}(\cE)|\bQ}^{1}$ for each $i\geq0$.
Then the differential $\partial_{1}\colon\Sym_{\bQ}(\cE)\otimes\cE \to \Sym_{\bQ}(\cE)$ is a $\Sym_{\bQ}(\cE)$-linear derivation satisfying $\partial_{1}(m_{i,1}) = e_i$ for all $i \geq 1$.
Therefore, it follows from Corollary~\ref{cor:stable_cohomology_first_homology_group_result} that
\begin{equation}\label{eq:iso_cohomology_tilde_H_Ker}
H_{\st}^{\odd}(\tilde{H}_{\bQ})\cong\Ker(\partial_{1}:\Omega_{\Sym_{\bQ}(\cE)|\bQ}^{1}\to\Omega_{\Sym_{\bQ}(\cE)|\bQ}^{0}),
\end{equation}
so that the truncation of \eqref{eq:differential_forms_resolution} provides an exact sequence of $\Sym_{\bQ}(\cE)$-modules
$$\xymatrix{\cdots\ar@{->}[r]^-{\partial_{3}} & \Omega_{\Sym_{\bQ}(\cE)|\bQ}^{2}\ar@{->}[r]^-{\partial_{2}} & H_{\st}^{\odd}(\tilde{H}_{\bQ})\ar@{->}[r] & 0.}$$
Then, it follows from the formula \eqref{eq:def_P_D_n} that the image of the differential $\partial_{2}$ exactly corresponds to the presentation by generators and relations of $\mathfrak{M}$, which ends the proof.
\end{proof}

\begin{rmk}[\emph{Interpretation in terms of functor homology.}]\label{rmk:functor_homology_interpretation2}
Contrary to the case of $\tilde{H}^{\vee}_{\bQ}$ (see Remark~\ref{rmk:functor_homology_interpretation}), Theorem~\ref{thm:stable_functor_homology} does not apply since the modules $\tilde{H}_{\bQ}$ define a \emph{covariant} twisted coefficient system.
\end{rmk}

\phantomsection
\addcontentsline{toc}{section}{References}
\renewcommand{\bibfont}{\normalfont\small}
\setlength{\bibitemsep}{0pt}
\printbibliography

@article {WahlRandal-Williams,
    AUTHOR = {Randal-Williams, Oscar and Wahl, Nathalie},
     TITLE = {Homological stability for automorphism groups},
   JOURNAL = {Adv. Math.},
  FJOURNAL = {Advances in Mathematics},
    VOLUME = {318},
      YEAR = {2017},
     PAGES = {534--626},
      ISSN = {0001-8708},
   MRCLASS = {20J05},
  MRNUMBER = {3689750},
       DOI = {10.1016/j.aim.2017.07.022},
       URL = {https://doi.org/10.1016/j.aim.2017.07.022},
}

@article {Harer,
    AUTHOR = {Harer, John L.},
     TITLE = {Stability of the homology of the mapping class groups of
              orientable surfaces},
   JOURNAL = {Ann. of Math. (2)},
  FJOURNAL = {Annals of Mathematics. Second Series},
    VOLUME = {121},
      YEAR = {1985},
    NUMBER = {2},
     PAGES = {215--249},
      ISSN = {0003-486X},
   MRCLASS = {57M99 (20F34)},
  MRNUMBER = {786348},
MRREVIEWER = {K. Vogtmann},
       DOI = {10.2307/1971172},
       URL = {https://doi-org.ezproxy.lib.gla.ac.uk/10.2307/1971172},
}

@article {HarerH_3,
    AUTHOR = {Harer, John},
     TITLE = {The third homology group of the moduli space of curves},
   JOURNAL = {Duke Math. J.},
  FJOURNAL = {Duke Mathematical Journal},
    VOLUME = {63},
      YEAR = {1991},
    NUMBER = {1},
     PAGES = {25--55},
      ISSN = {0012-7094},
   MRCLASS = {57M99 (14H15 20H10 32G15)},
  MRNUMBER = {1106936},
MRREVIEWER = {Darryl McCullough},
       DOI = {10.1215/S0012-7094-91-06302-7},
       URL = {https://doi-org.ezproxy.lib.gla.ac.uk/10.1215/S0012-7094-91-06302-7},
}

@article {KupersRandal-Williams,
    AUTHOR = {Kupers, Alexander and Randal-Williams, Oscar},
     TITLE = {On the cohomology of {T}orelli groups},
   JOURNAL = {Forum Math. Pi},
  FJOURNAL = {Forum of Mathematics. Pi},
    VOLUME = {8},
      YEAR = {2020},
     PAGES = {e7, 83},
   MRCLASS = {57S05 (11F75 18M05 20G05 55R40)},
  MRNUMBER = {4089394},
MRREVIEWER = {Nick Salter},
       DOI = {10.1017/fmp.2020.5},
       URL = {https://doi.org/10.1017/fmp.2020.5},
}

@article {KawazumiMorita1,
    AUTHOR = {Kawazumi, Nariya and Morita, Shigeyuki},
     TITLE = {The primary approximation to the cohomology of the moduli space of curves and cocycles for the Mumford-Morita-Miller classes},
   JOURNAL = {preprint},
  FJOURNAL = {},
    VOLUME = {},
      YEAR = {2001},
    NUMBER = {},
     PAGES = {},
      ISSN = {},
   MRCLASS = {},
  MRNUMBER = {},
MRREVIEWER = {},
       DOI = {},
       URL = {https://www.ms.u-tokyo.ac.jp/preprint/pdf/2001-13.pdf},
}

@article {KawazumiMorita,
    AUTHOR = {Kawazumi, Nariya and Morita, Shigeyuki},
     TITLE = {The primary approximation to the cohomology of the moduli
              space of curves and cocycles for the stable characteristic
              classes},
   JOURNAL = {Math. Res. Lett.},
  FJOURNAL = {Mathematical Research Letters},
    VOLUME = {3},
      YEAR = {1996},
    NUMBER = {5},
     PAGES = {629--641},
      ISSN = {1073-2780},
   MRCLASS = {14H10 (14F99 57R20)},
  MRNUMBER = {1418577},
MRREVIEWER = {Gheorghe Ionesei},
       DOI = {10.4310/MRL.1996.v3.n5.a6},
       URL = {https://doi.org/10.4310/MRL.1996.v3.n5.a6},
}

@article {KawazumitwistedMMM,
    AUTHOR = {Kawazumi, Nariya},
     TITLE = {A generalization of the {M}orita-{M}umford classes to extended
              mapping class groups for surfaces},
   JOURNAL = {Invent. Math.},
  FJOURNAL = {Inventiones Mathematicae},
    VOLUME = {131},
      YEAR = {1998},
    NUMBER = {1},
     PAGES = {137--149},
      ISSN = {0020-9910},
   MRCLASS = {57R20 (14H15 57M50)},
  MRNUMBER = {1489896},
MRREVIEWER = {Shigeyuki Morita},
       DOI = {10.1007/s002220050199},
       URL = {https://doi.org/10.1007/s002220050199},
}

@article {MoritaJacobianI,
    AUTHOR = {Morita, Shigeyuki},
     TITLE = {Families of {J}acobian manifolds and characteristic classes of
              surface bundles. {I}},
   JOURNAL = {Ann. Inst. Fourier (Grenoble)},
  FJOURNAL = {Universit\'{e} de Grenoble. Annales de l'Institut Fourier},
    VOLUME = {39},
      YEAR = {1989},
    NUMBER = {3},
     PAGES = {777--810},
      ISSN = {0373-0956},
   MRCLASS = {32G15 (55R40 57N05 57R20)},
  MRNUMBER = {1030850},
MRREVIEWER = {Ruth Charney},
       URL = {http://www.numdam.org/item?id=AIF_1989__39_3_777_0},
}

@article {MoritaJacobianII,
    AUTHOR = {Morita, Shigeyuki},
     TITLE = {Families of {J}acobian manifolds and characteristic classes of
              surface bundles. {II}},
   JOURNAL = {Math. Proc. Camb. Phil. Soc.},
  FJOURNAL = {Mathematical Proceedings of the Cambridge Philosophical Society},
    VOLUME = {105},
      YEAR = {1989},
    NUMBER = {3},
     PAGES = {79--101},
      ISSN = {0305-0041},
   MRCLASS = {},
  MRNUMBER = {},
MRREVIEWER = {},
       URL = {},
}

@incollection {stablecohomologyKawazumi,
    AUTHOR = {Kawazumi, Nariya},
     TITLE = {On the stable cohomology algebra of extended mapping class
              groups for surfaces},
 BOOKTITLE = {Groups of diffeomorphisms},
    SERIES = {Adv. Stud. Pure Math.},
    VOLUME = {52},
     PAGES = {383--400},
 PUBLISHER = {Math. Soc. Japan, Tokyo},
      YEAR = {2008},
   MRCLASS = {57M07 (57N05 57R20)},
  MRNUMBER = {2509717},
MRREVIEWER = {Thilo Kuessner},
       DOI = {10.2969/aspm/05210383},
       URL = {https://doi.org/10.2969/aspm/05210383},
}

@book {farbmargalit,
    AUTHOR = {Farb, Benson and Margalit, Dan},
     TITLE = {A primer on mapping class groups},
    SERIES = {Princeton Mathematical Series},
    VOLUME = {49},
 PUBLISHER = {Princeton University Press, Princeton, NJ},
      YEAR = {2012},
     PAGES = {xiv+472},
      ISBN = {978-0-691-14794-9},
   MRCLASS = {57M50 (20F36 20F65 57M07 57N05)},
  MRNUMBER = {2850125},
MRREVIEWER = {Stephen P. Humphries},
}

@incollection {Margalit,
    AUTHOR = {Margalit, Dan},
     TITLE = {Problems, questions, and conjectures about mapping class
              groups},
 BOOKTITLE = {Breadth in contemporary topology},
    SERIES = {Proc. Sympos. Pure Math.},
    VOLUME = {102},
     PAGES = {157--186},
 PUBLISHER = {Amer. Math. Soc., Providence, RI},
      YEAR = {2019},
   MRCLASS = {20F38},
  MRNUMBER = {3967367},
MRREVIEWER = {Nick Salter},
}

@unpublished{PSIIp,
     Title = {Polynomiality of surface braid and mapping class group representations},
    Author = {Palmer, Martin and Souli{\'e}, Arthur},
      Year = {2023},
      Note = {ArXiv: \href{https://arxiv.org/abs/2302.08827}{2302.08827}},
}

@book {Weibel,
    AUTHOR = {Weibel, Charles A.},
     TITLE = {An introduction to homological algebra},
    SERIES = {Cambridge Studies in Advanced Mathematics},
    VOLUME = {38},
 PUBLISHER = {Cambridge University Press, Cambridge},
      YEAR = {1994},
     PAGES = {xiv+450},
   MRCLASS = {18-01 (16-01 17-01 20-01 55Uxx)},
  MRNUMBER = {1269324},
MRREVIEWER = {Kenneth A. Brown},
       DOI = {10.1017/CBO9781139644136},
       URL = {https://doi-org.ezproxy.lib.gla.ac.uk/10.1017/CBO9781139644136},
}

@article {ChillingworthI,
    AUTHOR = {Chillingworth, D. R. J.},
     TITLE = {Winding numbers on surfaces. {I}},
   JOURNAL = {Math. Ann.},
  FJOURNAL = {Mathematische Annalen},
    VOLUME = {196},
      YEAR = {1972},
     PAGES = {218--249},
      ISSN = {0025-5831},
   MRCLASS = {57D25},
  MRNUMBER = {0300304},
MRREVIEWER = {C. W. Patty},
       DOI = {10.1007/BF01428050},
       URL = {https://doi.org/10.1007/BF01428050},
}

@article {ChillingworthII,
    AUTHOR = {Chillingworth, D. R. J.},
     TITLE = {Winding numbers on surfaces. {II}},
   JOURNAL = {Math. Ann.},
  FJOURNAL = {Mathematische Annalen},
    VOLUME = {199},
      YEAR = {1972},
     PAGES = {131--153},
      ISSN = {0025-5831,1432-1807},
   MRCLASS = {57A05 (55A05)},
  MRNUMBER = {321091},
MRREVIEWER = {C.\ W.\ Patty},
       DOI = {10.1007/BF01431419},
       URL = {https://doi.org/10.1007/BF01431419},
}

@article {MadsenWeiss,
    AUTHOR = {Madsen, Ib and Weiss, Michael},
     TITLE = {The stable moduli space of {R}iemann surfaces: {M}umford's
              conjecture},
   JOURNAL = {Ann. of Math. (2)},
  FJOURNAL = {Annals of Mathematics. Second Series},
    VOLUME = {165},
      YEAR = {2007},
    NUMBER = {3},
     PAGES = {843--941},
      ISSN = {0003-486X},
   MRCLASS = {14H10 (14F43 19D06 55P47)},
  MRNUMBER = {2335797},
MRREVIEWER = {Ulrike Tillmann},
       DOI = {10.4007/annals.2007.165.843},
       URL = {https://doi.org/10.4007/annals.2007.165.843},
}

@article {soulie3,
    AUTHOR = {Souli\'{e}, Arthur},
     TITLE = {Some computations of stable twisted homology for mapping class
              groups},
   JOURNAL = {Comm. Algebra},
  FJOURNAL = {Communications in Algebra},
    VOLUME = {48},
      YEAR = {2020},
    NUMBER = {6},
     PAGES = {2467--2491},
      ISSN = {0092-7872},
   MRCLASS = {20J05 (18M05 18M15 20F38 20J06)},
  MRNUMBER = {4107585},
MRREVIEWER = {Conchita Mart\'{\i}nez-P\'{e}rez},
       DOI = {10.1080/00927872.2020.1716981},
       URL = {https://doi-org.ezproxy.lib.gla.ac.uk/10.1080/00927872.2020.1716981},
}

@article {Trapp,
    AUTHOR = {Trapp, Rolland},
     TITLE = {A linear representation of the mapping class group $\mathscr{M}$ and the theory of winding numbers},
   JOURNAL = {Topology Appl.},
  FJOURNAL = {Topology and its Applications},
    VOLUME = {43},
      YEAR = {1992},
    NUMBER = {1},
     PAGES = {47--64},
      ISSN = {0166-8641},
   MRCLASS = {57M99 (55M25 55S37 57N05 57S99)},
  MRNUMBER = {1141372},
MRREVIEWER = {Darryl McCullough},
       DOI = {10.1016/0166-8641(92)90153-Q},
       URL = {https://doi.org/10.1016/0166-8641(92)90153-Q},
}

@book {Loday,
    AUTHOR = {Loday, Jean-Louis},
     TITLE = {Cyclic homology},
    SERIES = {Grundlehren der Mathematischen Wissenschaften [Fundamental
              Principles of Mathematical Sciences]},
    VOLUME = {301},
   EDITION = {Second},
      NOTE = {Appendix E by O. Ronco,
              Chapter 13 by the author in collaboration with Teimuraz
              Pirashvili},
 PUBLISHER = {Springer-Verlag, Berlin},
      YEAR = {1998},
     PAGES = {xx+513},
      ISBN = {3-540-63074-0},
   MRCLASS = {16E40 (13D03 17B55 18G60 19D55)},
  MRNUMBER = {1600246},
       DOI = {10.1007/978-3-662-11389-9},
       URL = {https://doi.org/10.1007/978-3-662-11389-9},
}

@book {BourbakiAlgHomol,
    AUTHOR = {Bourbaki, N.},
     TITLE = {\'{E}l\'{e}ments de math\'{e}matique. {A}lg\`ebre. {C}hapitre
              10. {A}lg\`ebre homologique},
      NOTE = {Reprint of the 1980 original [Masson, Paris; MR0610795]},
 PUBLISHER = {Springer-Verlag, Berlin},
      YEAR = {2007},
     PAGES = {viii+216},
      ISBN = {978-3-540-34492-6; 3-540-34492-6},
   MRCLASS = {18Gxx (00-01 01A75)},
  MRNUMBER = {2327161},
}

@article {graysonQuillen,
    AUTHOR = {Grayson, Daniel},
     TITLE = {Higher algebraic {$K$}-theory. {II} (after {D}aniel
              {Q}uillen)},
 BOOKTITLE = {Algebraic {$K$}-theory ({P}roc. {C}onf., {N}orthwestern
              {U}niv., {E}vanston, {I}ll., 1976)},
     PAGES = {217--240. Lecture Notes in Math., Vol. 551},
 PUBLISHER = {Springer, Berlin},
      YEAR = {1976},
   MRCLASS = {18F25},
  MRNUMBER = {0574096},
MRREVIEWER = {H. Bass},
}

@book {MacLane_Homology,
    AUTHOR = {Mac Lane, Saunders},
     TITLE = {Homology},
    SERIES = {Classics in Mathematics},
      NOTE = {Reprint of the 1975 edition},
 PUBLISHER = {Springer-Verlag, Berlin},
      YEAR = {1995},
     PAGES = {x+422},
      ISBN = {3-540-58662-8},
   MRCLASS = {18-02 (18Axx 18Cxx 18Gxx)},
  MRNUMBER = {1344215},
}

@article {soulieLMgeneralised,
    AUTHOR = {Souli{\'e}, Arthur},
     TITLE = {Generalized {L}ong-{M}oody functors},
   JOURNAL = {Algebr. Geom. Topol.},
  FJOURNAL = {Algebraic \& Geometric Topology},
    VOLUME = {22},
      YEAR = {2022},
    NUMBER = {4},
     PAGES = {1713--1788},
      ISSN = {},
   MRCLASS = {},
  MRNUMBER = {},
       DOI = {10.2140/agt.2022.22.1713},
       URL = {https://msp.org/agt/2022/22-4/index.xhtml},
}

@book {Brown,
    AUTHOR = {Brown, Kenneth S.},
     TITLE = {Cohomology of groups},
    SERIES = {Graduate Texts in Mathematics},
    VOLUME = {87},
      NOTE = {Corrected reprint of the 1982 original},
 PUBLISHER = {Springer-Verlag, New York},
      YEAR = {1994},
     PAGES = {x+306},
      ISBN = {0-387-90688-6},
   MRCLASS = {20J05 (20-02)},
  MRNUMBER = {1324339},
}

@article {Earle,
    AUTHOR = {Earle, Clifford J.},
     TITLE = {Families of {R}iemann surfaces and {J}acobi varieties},
   JOURNAL = {Ann. of Math. (2)},
  FJOURNAL = {Annals of Mathematics. Second Series},
    VOLUME = {107},
      YEAR = {1978},
    NUMBER = {2},
     PAGES = {255--286},
      ISSN = {0003-486X},
   MRCLASS = {32G15 (14H15 30A46)},
  MRNUMBER = {499328},
MRREVIEWER = {L. Keen},
       DOI = {10.2307/1971144},
       URL = {https://doi-org.ezproxy.lib.gla.ac.uk/10.2307/1971144},
}

@article {Morita_casson_invariant,
    AUTHOR = {Morita, Shigeyuki},
     TITLE = {Casson invariant, signature defect of framed manifolds and the
              secondary characteristic classes of surface bundles},
   JOURNAL = {J. Differential Geom.},
  FJOURNAL = {Journal of Differential Geometry},
    VOLUME = {47},
      YEAR = {1997},
    NUMBER = {3},
     PAGES = {560--599},
      ISSN = {0022-040X},
   MRCLASS = {57M50 (57R20)},
  MRNUMBER = {1617632},
MRREVIEWER = {Dave Auckly},
       URL = {http://projecteuclid.org.ezproxy.lib.gla.ac.uk/euclid.jdg/1214460550},
}

@incollection {FranjouPira,
    AUTHOR = {Franjou, Vincent and Pirashvili, Teimuraz},
     TITLE = {Stable {$K$}-theory is bifunctor homology (after {A}.
              {S}corichenko)},
 BOOKTITLE = {Rational representations, the {S}teenrod algebra and functor
              homology},
    SERIES = {Panor. Synth\`eses},
    VOLUME = {16},
     PAGES = {107--126},
 PUBLISHER = {Soc. Math. France, Paris},
      YEAR = {2003},
   MRCLASS = {19D55},
  MRNUMBER = {2117530},
}

@article {Morita_char_class,
    AUTHOR = {Morita, Shigeyuki},
     TITLE = {Characteristic classes of surface bundles},
   JOURNAL = {Invent. Math.},
  FJOURNAL = {Inventiones Mathematicae},
    VOLUME = {90},
      YEAR = {1987},
    NUMBER = {3},
     PAGES = {551--577},
      ISSN = {0020-9910},
   MRCLASS = {57R20 (32G15 57M99 58D15)},
  MRNUMBER = {914849},
MRREVIEWER = {Ronnie Lee},
       DOI = {10.1007/BF01389178},
       URL = {https://doi-org.ezproxy.lib.gla.ac.uk/10.1007/BF01389178},
}

@article {Miller,
    AUTHOR = {Miller, Edward Y.},
     TITLE = {The homology of the mapping class group},
   JOURNAL = {J. Differential Geom.},
  FJOURNAL = {Journal of Differential Geometry},
    VOLUME = {24},
      YEAR = {1986},
    NUMBER = {1},
     PAGES = {1--14},
      ISSN = {0022-040X},
   MRCLASS = {32G15 (57N05)},
  MRNUMBER = {857372},
MRREVIEWER = {Ronnie Lee},
       URL = {http://projecteuclid.org.ezproxy.lib.gla.ac.uk/euclid.jdg/1214440254},
}

@incollection {Mumford,
    AUTHOR = {Mumford, David},
     TITLE = {Towards an enumerative geometry of the moduli space of curves},
 BOOKTITLE = {Arithmetic and geometry, {V}ol. {II}},
    SERIES = {Progr. Math.},
    VOLUME = {36},
     PAGES = {271--328},
 PUBLISHER = {Birkh\"{a}user Boston, Boston, MA},
      YEAR = {1983},
   MRCLASS = {14H10 (14C15)},
  MRNUMBER = {717614},
MRREVIEWER = {Werner Kleinert},
}

@book {Milnor,
    AUTHOR = {Milnor, John},
     TITLE = {Introduction to algebraic {$K$}-theory},
      NOTE = {Annals of Mathematics Studies, No. 72},
 PUBLISHER = {Princeton University Press, Princeton, N.J.; University of
              Tokyo Press, Tokyo},
      YEAR = {1971},
     PAGES = {xiii+184},
   MRCLASS = {18F25 (12A65)},
  MRNUMBER = {0349811},
MRREVIEWER = {J.-P. Jouanolou},
}

@article {Looijenga,
    AUTHOR = {Looijenga, Eduard},
     TITLE = {Stable cohomology of the mapping class group with symplectic
              coefficients and of the universal {A}bel-{J}acobi map},
   JOURNAL = {J. Algebraic Geom.},
  FJOURNAL = {Journal of Algebraic Geometry},
    VOLUME = {5},
      YEAR = {1996},
    NUMBER = {1},
     PAGES = {135--150},
      ISSN = {1056-3911},
   MRCLASS = {14H15 (14C30 14F25)},
  MRNUMBER = {1358038},
}

@incollection {Ivanov,
    AUTHOR = {Ivanov, Nikolai V.},
     TITLE = {On the homology stability for {T}eichm\"{u}ller modular groups:
              closed surfaces and twisted coefficients},
 BOOKTITLE = {Mapping class groups and moduli spaces of {R}iemann surfaces
              ({G}\"{o}ttingen, 1991/{S}eattle, {WA}, 1991)},
    SERIES = {Contemp. Math.},
    VOLUME = {150},
     PAGES = {149--194},
 PUBLISHER = {Amer. Math. Soc., Providence, RI},
      YEAR = {1993},
   MRCLASS = {57N05 (20F38 30F60 32G15 57M99)},
  MRNUMBER = {1234264},
MRREVIEWER = {Darryl McCullough},
       DOI = {10.1090/conm/150/01290},
       URL = {https://doi-org.ezproxy.lib.gla.ac.uk/10.1090/conm/150/01290},
}

@article {CCS,
    AUTHOR = {Callegaro, F. and Cohen, F. R. and Salvetti, M.},
     TITLE = {The cohomology of the braid group {$B_3$} and of {${\rm
              SL}_2(\Bbb Z)$} with coefficients in a geometric
              representation},
   JOURNAL = {Q. J. Math.},
  FJOURNAL = {The Quarterly Journal of Mathematics},
    VOLUME = {64},
      YEAR = {2013},
    NUMBER = {3},
     PAGES = {847--889},
      ISSN = {0033-5606},
   MRCLASS = {20J06 (20F36 20G10)},
  MRNUMBER = {3094502},
MRREVIEWER = {Marian F. Anton},
       DOI = {10.1093/qmath/hat027},
       URL = {https://doi-org.ezproxy.lib.gla.ac.uk/10.1093/qmath/hat027},
}

@article {Morita_char_class_bull,
    AUTHOR = {Morita, Shigeyuki},
     TITLE = {Characteristic classes of surface bundles},
   JOURNAL = {Bull. Amer. Math. Soc. (N.S.)},
  FJOURNAL = {American Mathematical Society. Bulletin. New Series},
    VOLUME = {11},
      YEAR = {1984},
    NUMBER = {2},
     PAGES = {386--388},
      ISSN = {0273-0979},
   MRCLASS = {55R40 (32G15 57N05 57R20)},
  MRNUMBER = {752805},
       DOI = {10.1090/S0273-0979-1984-15321-7},
       URL = {https://doi-org.ezproxy.lib.gla.ac.uk/10.1090/S0273-0979-1984-15321-7},
}

@article {Hain,
    AUTHOR = {Hain, Richard},
     TITLE = {Johnson homomorphisms},
   JOURNAL = {EMS Surv. Math. Sci.},
  FJOURNAL = {EMS Surveys in Mathematical Sciences},
    VOLUME = {7},
      YEAR = {2020},
    NUMBER = {1},
     PAGES = {33--116},
      ISSN = {2308-2151},
   MRCLASS = {57K20 (14C30 17B62)},
  MRNUMBER = {4195745},
MRREVIEWER = {Harry Tamvakis},
       DOI = {10.4171/emss/36},
       URL = {https://doi.org/10.4171/emss/36},
}

@article {RWautfreegroups,
    AUTHOR = {Randal-Williams, Oscar},
     TITLE = {Cohomology of automorphism groups of free groups with twisted
              coefficients},
   JOURNAL = {Selecta Math. (N.S.)},
  FJOURNAL = {Selecta Mathematica. New Series},
    VOLUME = {24},
      YEAR = {2018},
    NUMBER = {2},
     PAGES = {1453--1478},
      ISSN = {1022-1824},
   MRCLASS = {20F28 (20J06 57R20)},
  MRNUMBER = {3782426},
MRREVIEWER = {Valeriy G. Bardakov},
       DOI = {10.1007/s00029-017-0311-0},
       URL = {https://doi.org/10.1007/s00029-017-0311-0},
}

@article {Randal-WilliamsJEMS,
    AUTHOR = {Randal-Williams, Oscar},
     TITLE = {Resolutions of moduli spaces and homological stability},
   JOURNAL = {J. Eur. Math. Soc. (JEMS)},
  FJOURNAL = {Journal of the European Mathematical Society (JEMS)},
    VOLUME = {18},
      YEAR = {2016},
    NUMBER = {1},
     PAGES = {1--81},
      ISSN = {1435-9855},
   MRCLASS = {55R40 (20J06 57M07 57R15 57R50)},
  MRNUMBER = {3438379},
MRREVIEWER = {Yusuke Kuno},
       DOI = {10.4171/JEMS/583},
       URL = {https://doi.org/10.4171/JEMS/583},
}

@article {Boldsen,
    AUTHOR = {Boldsen, S\o ren K.},
     TITLE = {Improved homological stability for the mapping class group
              with integral or twisted coefficients},
   JOURNAL = {Math. Z.},
  FJOURNAL = {Mathematische Zeitschrift},
    VOLUME = {270},
      YEAR = {2012},
    NUMBER = {1-2},
     PAGES = {297--329},
      ISSN = {0025-5874},
   MRCLASS = {57M50 (55T05 57N05)},
  MRNUMBER = {2875835},
MRREVIEWER = {Thomas Koberda},
       DOI = {10.1007/s00209-010-0798-y},
       URL = {https://doi.org/10.1007/s00209-010-0798-y},
}

@article {GalatiusKupersRW,
    AUTHOR = {Galatius, S\o ren and Kupers, Alexander and Randal-Williams,
              Oscar},
     TITLE = {{$E_2$}-cells and mapping class groups},
   JOURNAL = {Publ. Math. Inst. Hautes \'{E}tudes Sci.},
  FJOURNAL = {Publications Math\'{e}matiques. Institut de Hautes \'{E}tudes
              Scientifiques},
    VOLUME = {130},
      YEAR = {2019},
     PAGES = {1--61},
      ISSN = {0073-8301},
   MRCLASS = {20J05 (20F38 57K20)},
  MRNUMBER = {4028513},
MRREVIEWER = {Nick Salter},
       DOI = {10.1007/s10240-019-00107-8},
       URL = {https://doi-org.ezproxy.lib.gla.ac.uk/10.1007/s10240-019-00107-8},
}

@unpublished{KawazumiSoulieII,
    AUTHOR = {Kawazumi, Nariya and Souli{\'e}, Arthur},
     TITLE = {Stable twisted homology of the mapping class groups in the exterior powers of the unit tangent bundle (co)homology},
      YEAR = {2023},
      NOTE = {ArXiv: \href{https://arxiv.org/abs/2311.01791}{2311.0179}},
}

@article {Kuno,
    AUTHOR = {Kuno, Yusuke},
     TITLE = {A combinatorial formula for {E}arle's twisted 1-cocycle on the
              mapping class group {$\mathcal{M}_{g,*}$}},
   JOURNAL = {Math. Proc. Cambridge Philos. Soc.},
  FJOURNAL = {Mathematical Proceedings of the Cambridge Philosophical
              Society},
    VOLUME = {146},
      YEAR = {2009},
    NUMBER = {1},
     PAGES = {109--118},
      ISSN = {0305-0041},
   MRCLASS = {30F20 (20F28 57M07)},
  MRNUMBER = {2461871},
MRREVIEWER = {Alexander Zvonkin},
       DOI = {10.1017/S0305004108001680},
       URL = {https://doi.org/10.1017/S0305004108001680},
}

@article {DV1,
    AUTHOR = {Djament, Aur\'{e}lien and Vespa, Christine},
     TITLE = {Sur l'homologie des groupes orthogonaux et symplectiques \`a
              coefficients tordus},
   JOURNAL = {Ann. Sci. \'{E}c. Norm. Sup\'{e}r. (4)},
  FJOURNAL = {Annales Scientifiques de l'\'{E}cole Normale Sup\'{e}rieure. Quatri\`eme
              S\'{e}rie},
    VOLUME = {43},
      YEAR = {2010},
    NUMBER = {3},
     PAGES = {395--459},
      ISSN = {0012-9593},
   MRCLASS = {20G10 (19D45 20G25 20G30)},
  MRNUMBER = {2667021},
MRREVIEWER = {Boris \`E. Kunyavski\u{\i}},
       DOI = {10.24033/asens.2125},
       URL = {https://doi-org.ezproxy.lib.gla.ac.uk/10.24033/asens.2125},
}

@article {Galatius,
    AUTHOR = {Galatius, S\o ren},
     TITLE = {Mod {$p$} homology of the stable mapping class group},
   JOURNAL = {Topology},
  FJOURNAL = {Topology. An International Journal of Mathematics},
    VOLUME = {43},
      YEAR = {2004},
    NUMBER = {5},
     PAGES = {1105--1132},
      ISSN = {0040-9383},
   MRCLASS = {57M50 (55P47)},
  MRNUMBER = {2079997},
       DOI = {10.1016/j.top.2004.01.011},
       URL = {https://doi-org.ezproxy.lib.gla.ac.uk/10.1016/j.top.2004.01.011},
}

@article {Johnsonabelian,
    AUTHOR = {Johnson, Dennis},
     TITLE = {An abelian quotient of the mapping class group {${\cal
              I}\sb{g}$}},
   JOURNAL = {Math. Ann.},
  FJOURNAL = {Mathematische Annalen},
    VOLUME = {249},
      YEAR = {1980},
    NUMBER = {3},
     PAGES = {225--242},
      ISSN = {0025-5831,1432-1807},
   MRCLASS = {57N05 (20F38)},
  MRNUMBER = {579103},
MRREVIEWER = {W.\ D.\ Neumann},
       DOI = {10.1007/BF01363897},
       URL = {https://doi.org/10.1007/BF01363897},
}

@incollection {Johnsonsurvey,
    AUTHOR = {Johnson, Dennis},
     TITLE = {A survey of the {T}orelli group},
 BOOKTITLE = {Low-dimensional topology ({S}an {F}rancisco, {C}alif., 1981)},
    SERIES = {Contemp. Math.},
    VOLUME = {20},
     PAGES = {165--179},
 PUBLISHER = {Amer. Math. Soc., Providence, RI},
      YEAR = {1983},
      ISBN = {0-8218-5016-4},
   MRCLASS = {57N05 (14H15 32G15 57-02)},
  MRNUMBER = {718141},
MRREVIEWER = {J.\ S.\ Birman},
       DOI = {10.1090/conm/020/718141},
       URL = {https://doi.org/10.1090/conm/020/718141},
}

\noindent {Nariya Kawazumi, \itshape Department of Mathematical Sciences, 
University of Tokyo, 3-8-1 Komaba, Meguro-ku, Tokyo 153-
8914, Japan}. \noindent {Email address: \tt kawazumi@ms.u-tokyo.ac.jp}

\noindent {Arthur Souli{\'e}, \itshape Normandie Univ., UNICAEN, CNRS, LMNO, 14000 Caen, France.} \noindent {Email address: \tt artsou@hotmail.fr, arthur.soulie@unicaen.fr}

\end{document}